\pdfoutput=1

\documentclass[11pt, a4paper]{article}

\usepackage[draft]{sty/dl-en}
\usepackage{sty/config}

\DeclareMathOperator{\Exp}{{\mathrm{Exp}}}

\DeclareMathOperator{\TopExp}{{\mathrm{Exp}}_{\mathrm T}}

\DeclareMathOperator{\MetricExp}{\mathrm{Exp}_{\mathrm M}}

\DeclareMathOperator{\Cone}{{\mathrm{C}}}

\newcommand{\Exit}[1]{{\mathsf{Exit}_{#1}}}

\newcommand{\PointedOmega}{\OrdinalOmega_\ast}

\newcommand{\OrdinalOmega}{\omegaup}

\newcommand{\Op}{^\mathrm{op}}

\newcommand{\PiZero}{\piup_0}

\newcommand{\TOne}{\mathsf{T}_1}

\newcommand{\Simplex}[1]{\Deltaup^{#1}}
\newcommand{\Category}[1]{\mathcal #1}

\newcommand{\CW}{\textnormal{CW}}
\newcommand{\Sh}{\mathsf{Sh}}
\newcommand{\Hyp}{\mathsf{Hyp}}
\newcommand{\Fun}{\mathsf{Fun}}
\newcommand{\Spaces}{\mathsf S}
\newcommand{\Sheaf}[1]{\mathcal #1}
\newcommand{\LocallyConstantSheaves}{\mathsf{Sh}_\mathrm{loc}}
\newcommand{\Hypercompletion}[1]{\widehat{#1}}
\newcommand{\LocallyConstantHypersheaves}{\mathsf{Hyp}_\mathrm{loc}}
\newcommand{\LocallyHyperConstantHypersheaves}{\mathsf{Hyp}_\mathrm{loc-hyp}}

\newcommand{\ConstructibleSheaves}[1]{\mathsf{Sh}_{#1}}
\newcommand{\HyperConstructibleSheaves}[1]{\mathsf{Sh}_{#1\text{-}\mathrm{hyp}}}
\newcommand{\ConstructibleHypersheaves}[1]{\mathsf{Hyp}_{#1}}

\newcommand{\HyperConstructibleHypersheaves}[1]{\mathsf{Hyp}_{#1\text{-}\mathrm{hyp}}}
\newcommand{\Hyperpullback}[1]{\widehat{#1^\ast}}
\newcommand{\FinalMap}{\piup}

\newcommand{\Sing}{\operatorname{Sing}}
\newcommand{\GeometricRealisation}[1]{|#1|}

\newcommand{\Slice}[2]{{#1}_{/#2}}

\newcommand{\Topos}[1]{\mathcal #1}
\newcommand{\FinalSheaf}{\mathbf 1}
\newcommand{\Opens}[1]{\mathfrak{O}(#1)}

\newcommand{\Etale}[1]{\mathfrak{E}(#1)}

\newcommand{\SSets}{\mathsf{Set}_{\Deltaup}}
\newcommand{\LocallyConstantRealisation}{\psiup}
\newcommand{\ConstructibleRealisation}{\Psiup}
\newcommand{\DD}{\mathfrak D}

\newcommand{\DDO}{\DD_{\OrdinalOmega}}

\newcommand{\IsAdjointTo}{\mathrel{\dashv}}
\newcommand{\FSigma}{\mathrm{F}_\sigmaup}
\newcommand{\Nerve}{\mathrm{N}}
\newcommand{\Presheaves}{\mathsf{PSh}}

\newcommand{\SliceExit}{\mathbf A}
\newcommand{\CC}{\mathfrak C}
\newcommand{\Un}{\mathrm{Un}}
\newcommand{\Yoneda}{\deltaup}
\newcommand{\Edges}{\mathrm{E}}
\newcommand\Restriction[2]{{
  \left.\kern-\nulldelimiterspace 
  #1 
  \vphantom{\big|} 
  \right|_{#2} 
  }}

\newcommand{\UnskipRef}[1]{\unskip~\textnormal{[\ref{#1}]}}
\newcommand{\DoubleUnskipRef}[2]{\unskip~[\ref{#1}, \ref{#2}]}
\newcommand{\TripleUnskipRef}[3]{\unskip~[\ref{#1}, \ref{#2}, \ref{#3}]}
\newcommand{\SectionRef}[1]{\unskip~[\hyperref[#1]{\S\thinspace\ref{#1}}]}

\theoremstyle{definition}
	\newtheorem{warning}[theorem]{Warning}
	
	\newtheorem{argument}[theorem]{Argument}

\usepackage{pgfplots}
\pgfplotsset{compat=1.17}
\pgfdeclarelayer{pre main}

\begin{document}

\maketitle

\footnotefirstpage

\begin{abstract}
	The goal of this article is to extend a theorem of Lurie
	\[
		\ConstructibleSheaves A (X)
		\IsCanonicallyIsomorphicTo
		\Fun(\Exit A(X), \Spaces)
	\]
	representing constructible sheaves with values in \( \Spaces \),
	the ∞-category of spaces, on a stratified space
	\( X \) with poset of strata \( A \),
	as functors from the exit paths ∞-category
	\( \Exit A (X) \) to \( \Spaces \).
	Lurie's representation theorem works provided
	\( A \) satisfy the ascending chain condition.
	This typically rules out infinite dimensional examples
	of stratified space.

	Building on it and with the help
	of a stratified homotopy invariance theorem
	from Haine,
	we show that when \( X \) is a nice enough \( A \)-stratified
	space and when \( A \) is itself stratified
	\( A_{\leq 0} \subset A_{\leq 1}
	\subset \cdots \subset A \) by posets satisfying the ascending
	chain condition,
	\[
		\ConstructibleHypersheaves A (X)
		\IsCanonicallyIsomorphicTo
		\Fun(\Exit A(X), \Spaces)
	\]
	the ∞-category of \( A \)-constructible hypersheaves
	on \( X \) is represented by functors from the
	exit paths ∞-category of \( X \).

	There are two types of nice stratified spaces on which
	this extended representation theorem applies:
	conically stratified spaces
	and spaces that are sequential colimits of conically stratified
	spaces.
	Examples of application include the
	metric and the topological exponentials of a Fréchet
	manifold,
	locally countable simplicial complexes and more generally,
	locally countable
	cylindrically normal \( \CW \)-complexes.
\end{abstract}

\TocWithoutIntroSection

When a topological space is nice enough, its
category of locally constant sheaves of sets
is equivalent to the category of
representations of its fundamental groupoid.
The fundamental groupoid \( X \) has objects the points of \( X \)
and arrows, the homotopy classes of continuous paths
between two points in \( X \) up to homotopy.
Going further, Lurie has shown that the ∞\=/category of locally
constant sheaves of spaces is equivalent to the ∞\=/category of
representations of \( \Sing(X) \), the simplicial
set of maps \( \Simplex n \to X \), which is a model
for the fundamental ∞-groupoid of \( X \)
\cite[A.2.15]{arXiv:0911.0018}.

This representation theorem can be further extended
to stratified spaces.
When \( X \) is a stratified space with poset of strata \( A \), one
can consider \( A \)\=/constructible sheaves on \( X \): sheaves
whose restriction to each stratum \( X_a \) is locally constant.
In order to represent those constructible sheaves, the simplicial
set \( \Sing(X) \) needs to be adapted to take into account
the stratification of \( X \).
Following an idea of Treumann
\cite{doi:10.1112/s0010437x09004229}, Lurie
considered the simplicial subset
\( \Exit A (X) \subset \Sing(X) \) where paths are only allowed
to immediately escape a deeper stratum and never return.
When the stratification is conical this simplicial subset
is an ∞-category, the \emph{exit paths ∞-category} of \( X \).
With such a setup, the representation theorem for \( A \)-constructible
sheaves on \( X \) holds \emph{provided \( A \) satisfies the ascending
chain condition}.

\begin{definition*}[(Ascending chain condition)]
	A poset \( A \) is said to satisfy the ascending chain
	condition if \( A \) does not admit
	a chain \( a_0 < a_1 < \cdots \) of infinite length.
\end{definition*}

This condition typically excludes stratified spaces of infinite
dimension.
For example, one may think of infinite dimensional
simplicial complexes which are not
locally finite or of any space \( X \) with a filtration by
dimension
\[
	X_{\leq 0} \subset X_{\leq 1} \subset \cdots
	\subset X_{\leq n} \subset \cdots
	\subset X
\]
which happens every time \( X \) is for example the colimit
\( \varinjlim_{n < \OrdinalOmega} X_{\leq n} \).
Our goal is to extend the representation theorem to a large
class of posets \( A \) that do not satisfy the ascending chain
condition and which includes in particular the poset
\( \OrdinalOmega = \{0 < 1 < 2 < \cdots \} \).

There are two obstacles to a generalisation of the representation
theorem.
The first one has to do with hypercompleteness of sheaves, a phenomenon
that starts appearing only in the ∞-world.
We shall dedicate a section to the differences between sheaves and
hypersheaves
\SectionRef{Section: hypersheaves}.
Having built a continuous map
\[
	\Fun(\Exit A (X), \Spaces)
	\longrightarrow
	\ConstructibleSheaves A (X)
\]
because every functor in \( \Fun(\Exit A (X), \Spaces) \)
is the limit of its
truncation tower, it follows that
its image must be a hypersheaf.
There is thus no hope of representing all \( A \)-constructible sheaves
in general and we shall instead focus on
the full subcategory of \emph{\( A \)-constructible hypersheaves}.
Notice that when \( A \) satisfies the ascending chain condition,
all \( A \)-constructible sheaves are already hypersheaves.

Constructible hypersheaves have already been used by Lurie to describe
the equivalence between locally constant factorisation algebras
on a finite dimensional manifold \( M \)
and \( \mathsf{E}_M \)-algebras.
A key tool in the proof of Lurie is the use of
the metric exponential of
\( M \), the metric space of finite subsets
of \( M \), which is naturally stratified by the cardinality
of the subsets.
Lurie notes that he had to add an hypercompletion hypothesis
on the constructible sheaves on the exponential because
the stratifying poset did not satisfy the ascending chain condition
\cite[3.3.12]{arXiv:0911.0018}.
On the other side of the equation, Cepek has shown that
the combinatoric of the exit paths ∞-category of the exponential
of \( \Reals^n \) is also
related to the one of \( \mathsf{E}_n \)-algebras
\cite{arXiv:1910.11980}.
The study of the exponential of a manifold is a major
motivation for extending the representation theorem and we shall
give more details about this particular example
at the end of the article
\SectionRef{Section: exponentials}.

The second obstacle has to do with the proof of the representation
theorem itself.
It uses an induction on the depth of the stratification and
a poset \( A \) admits a depth function if and only if
it satisfies the ascending chain condition.
To circumvent this issue, we shall make use of the functoriality
of the equivalence of ∞-categories in the representation theorem.
We shall then work with posets \( A \) which are themselves stratified
\[
	A_{\leq 0} \subset A_{\leq 1} \subset \cdots \subset A_{\leq n}
	\subset \cdots \subset A
\]
by posets satisfying the ascending chain condition.
In such a case, the stratified space \( X \) inherits a filtration
\[
	X_{\leq 0} \subset X_{\leq 1} \subset \cdots \subset X_{\leq n}
	\subset \cdots \subset X
\]
by closed stratified subspaces.
We are thus considering posets that are strict ind-objects
in the category of posets with the ascending chain condition.
In another perspective, Barwick, Glasman and Haine have considered
spaces (and ∞-toposes) stratified over
profinite posets \cite{arXiv:1807.03281}.

The functoriality of the exit paths ∞-category is easily adressed
as it commutes with filtered colimits, since simplicies
in \( \Exit A (X) \) are only allowed to visit a finite
number of strata.
Undertaking the functoriality of
the ∞\=/category of constructible hypersheaves is the real task here.
We shall show that the canonical dévissage map
\[
	\textstyle
	\ConstructibleHypersheaves A (X)
	\longrightarrow \varprojlim_{n < \OrdinalOmega}
	\ConstructibleSheaves {A_{\leq n}} (X_{\leq n})
\]
is an equivalence in two special cases: when \( X \) is conically
\( A \)-stratified and when the topology on \( X \) coincides
with the colimit topology
\( \varinjlim_{n < \OrdinalOmega} X_{\leq n} \).
In particular, one can replace the topology of a conically
stratified \( X \) with the colimit topology and keep the same
∞-category of constructible hypersheaves.
This is coherent with the fact that the ∞-category of exit paths
does not see the global topology of \( X \), as every map
\( \Simplex n \to X \) in \( \Exit A (X) \)
is required to visit only a finite number of strata.
These two sets of conditions are usually incompatible as explained
in an impossibility theorem
\UnskipRef{Theorem: impossibility}.
We then obtain two versions of the representation theorem.

\begin{theorem*}[{\UnskipRef{Representation, conical case}}
	(conical case)]
	Let \( X \) be a paracompact
	conically
	\( A \)-stratified space such
	that each stratum of \( X \) be locally of singular shape
	and \( A \) be \( \OrdinalOmega \)-stratified with
	\( A_{\leq n} \) satisfying the ascending chain condition,
	for each \( n < \OrdinalOmega \).
	Then the ∞-category of exit paths
	\( \Exit A (X) \)
	represents
	\[
		\Fun(\Exit A (X), \Spaces)
		\IsCanonicallyIsomorphicTo
		\ConstructibleHypersheaves A (X)
	\]
	the
	∞-category of \( A \)\=/constructible hypersheaves on \( X \).
\end{theorem*}

\begin{theorem*}[{\UnskipRef{Representation, colimit case}}
	(colimit case)]
	Let \( X \) be an \( A \)-stratified space,
	colimit of a sequence of closed stratified embeddings
	of paracompact conically stratified spaces over
	posets satisfying the ascending chain condition
	and whose strata are locally of singular shape.
	The ∞-category of exit paths
	\( \Exit A (X) \) represents
	\[
		\Fun(\Exit A (X), \Spaces)
		\IsCanonicallyIsomorphicTo
		\ConstructibleHypersheaves A (X)
		\IsCanonicallyIsomorphicTo
		\ConstructibleSheaves A (X)
	\]
	the ∞-category of \( A \)\=/constructible sheaves
	and all \( A \)\=/constructible sheaves on \( X \)
	are hypersheaves.
\end{theorem*}

The dévissage theorem in the conical case relies heavily on a
stratified homotopy invariance theorem from Haine
\cite[2.3]{arXiv:2010.06473}.
However, it is neither constructible sheaves nor
constructible hypersheaves that are invariant but
\emph{hyperconstructible hypersheaves}.
This other notion of constructibility stems from the difference
in functoriality between sheaves and hypersheaves.

For this reason, we shall dedicate the first
section to the definitions of all
the types of sheaves and constructibilities
that we have mentioned so far.
In the second section, we shall see that
for general types of spaces
constructible
sheaves, constructible hypersheaves and hyperconstructible hypersheaves
do coincide.
The third section is dedicated to the exit paths ∞-category
and the proof of the extended representation theorem.
The last section shall present some examples that this
new representation theorem allows us to consider: the metric
and the topological exponentials of a Fréchet
manifold, locally countable simplicial
complexes and more generally locally countable
cylindrically normal \( \CW \)-complexes.

\section{Different notions of constructibility}

We start by presenting the different characters at play:
stratified spaces, hyperconstructible hypersheaves,
constructible hypersheaves and constructible sheaves.

\subsection{Stratified spaces}

There exists many different non-equivalent notions of stratified
spaces in topology.
Here we shall use a very general one, following Lurie
\cite[A.5.1]{arXiv:0911.0018}.

\begin{definition}[(Stratified space)]
	A stratified topological space
	is the data of continuous map \( f \From X \to A \)
	where \( A \) is a poset viewed as a topological space
	by defining \( U \subset A \) to be open if
	and only if it is closed upwards.

	A morphism of stratified spaces is a commutative square
	\[
		\begin{tikzcd}[ampersand replacement=\&]
			X
			\arrow[r, ""]
			\arrow[d, "", swap]
			\& Y
			\arrow[d, ""] \\
			A
			\arrow[r, "", swap]
			\& B 
		\end{tikzcd}
	\]
	where the top map is continuous and the bottom map
	is a poset map.

	For each \( a \in A \), we shall let \( X_a \)
	denote the \( a \)-stratum, preimage \( f^{-1}(a) \).
	We shall denote by \( s_a \), the embedding
	\( X_a \subset X \).
	We shall also let \( X_{\leq a} \) denote the fibre product
	\( X \times_{A_{\leq a}} A \).

\end{definition}

\begin{remark}
	If \( A \) is itself stratified by a poset \( \Lambda \),
	we thus get two maps \( X \to A \to \Lambda \)
	and for every \( \lambda \in \Lambda \),
	\( X_{\leq \lambda} \) is naturally stratified
	over \( A_{\leq \lambda} \).
\end{remark}

We shall mainly be interested in stratified maps over an identity
morphism \( A = B \).
Indeed we shall essentially focus on stratified homotopy equivalences.
Also, if \( A \) is a subposet of \( B \), then one can see \( X \)
as being \( B \)-stratified without loss of generality.

\begin{definition}[(Stratified homotopy equivalence)]
	Let \( X \to A \) and \( Y \to A \) be two \( A \)-stratified
	spaces.
	An \( A \)-stratified homotopy between two
	\( A \)-stratified map \( f, g \From X \to Y \) is an
	\( A \)-stratified map \( h \From [0,1] \times X \to Y \)
	such that \( h(0, -) = f \) and \( h(1,-) = g \).

	We shall say that an \( A \)-stratified map
	\( f \From X \to Y \)
	is an \( A \)-stratified homotopy equivalence if
	there exists an \( A \)-stratified map
	\( g \From Y \to X \) and \( A \)-stratified homotopies
	between \( f\circ g \) and \( \id_Y \) on one hand,
	and between \( g \circ f \) and \( \id_X \) on the other hand.
\end{definition}

\subsection{Recollections on hypersheaves}
\label{Section: hypersheaves}

In the classical theory of sheaves of sets, it is well known that
isomorphisms can be detected on stalks.
This is no longer true for sheaves in the ∞-categorical world and this
leads to a separation into two equally interesting objects: sheaves
and hypersheaves.
The main difference here to which we shall pay special attention is the
difference in functoriality: for sheaves one uses pullbacks but
for hypersheaves one needs to use hyperpullbacks.
There are several equivalent definitions of hypersheaves
\cite[A.1.9]{arXiv:0911.0018}.
Since we are only interested in the case of hypersheaves on topological
spaces, we shall choose the most convenient definition.

\begin{definition}[(Hypersheaf)]
	Given a topological space \( X \), we shall denote
	by \( \Sh(X) \) the ∞-category of sheaves (of spaces) on
	\( X \) and by \( \Hyp(X) \) the full subcategory of
	hypersheaves: sheaves that are local with respect to
	maps \( \Sheaf F \to \Sheaf G  \)
	inducing an equivalence \( \Sheaf F_x \to \Sheaf G_x \)
	on stalks,
	for every point \( x \in X \).

	The inclusion \( \Hyp(X) \subset \Sh(X) \) admits a left
	exact reflector, which sends a sheaf \( \Sheaf F \)
	to its hypercompletion \( \Hypercompletion {\Sheaf F} \).
	It is such that the canonical map
	\( \Sheaf F_x \to \Hypercompletion {\Sheaf F}_x \)
	is an equivalence for every point \( x \in X \).
\end{definition}

Categories of sheaves admit the following functoriality: if
\( f \From X \to Y \) is a continuous map, it induces an
adjunction
\[
	\begin{tikzcd}[ampersand replacement=\&, column sep = huge]
		\Sh(X) \arrow[r, shift right=2, swap, "f_\ast"]
		\&
		\Sh(Y) \arrow[l, shift right=2, swap, "f^\ast"]
	\end{tikzcd}
\]
between the ∞-categories of sheaves on \( X \) and sheaves on \( Y \).
The right adjoint \( f_\ast \) preserves hypersheaves but not
the left adjoint \( f^\ast \).

\begin{definition}[(Hyperpullback)]
	Let \( f \From X \to Y \) be a continuous map and let
	\[
		\begin{tikzcd}
			\Sh(Y) \rar["\Hyperpullback f"]
				& \Hyp(X)
		\end{tikzcd}
	\]
	denote the hyperpullback along \( f \), obtained by
	first using \( f^\ast \) and then hypercompleting.
\end{definition}

By construction, we obtain an adjunction
\[
	\begin{tikzcd}[ampersand replacement=\&, column sep = huge]
		\Hyp(X) \arrow[r, shift right=2, swap, "f_\ast"]
		\&
		\Hyp(Y)
		\arrow[l, shift right=2, swap, "\Hyperpullback f"]
	\end{tikzcd}
\]
similar to the sheaf case.

\begin{remark}
	\label{Proposition: hyperpullback}
	Given two continuous maps \( f \From X \to Y \)
	and \( g \From Y \to Z \) and a hypersheaf \( \Sheaf H \)
	on \( Z \), the canonical map
	\[
		\Hyperpullback {(gf)} \Sheaf H
		\longrightarrow \Hyperpullback f \Hyperpullback g
		\Sheaf H
	\]
	is an equivalence.
	This stem from the fact that
	\( (gf)_\ast = g_\ast f_\ast \).
\end{remark}

In addition,
there is a useful case where hyperpullbacks and pullbacks coincide
for hypersheaves.

\begin{lemma}[{\cite[A.3.6]{arXiv:0911.0018}}]
	Let \( f \From X \to Y \) be a continuous map.
	Assume \( f^\ast \) admits a left adjoint \( f_! \), then
	for every hypersheaf \( \Sheaf H \)
	the canonical map
	\[
		f^\ast \Sheaf H \longrightarrow \Hyperpullback f
		\Sheaf H
	\]
	is an equivalence.

	This happens for example, when \( f \) is an open embedding.
\end{lemma}

\subsection{Locally (hyper)constant (hyper)sheaves}

Because sheaves and hypersheaves do not share the same functoriality,
there are two possible candidates to extend the traditional
notion of locally constant sheaves of sets
to the ∞-category world: locally constant sheaves
and locally hyperconstant hypersheaves.

Let \( X \) be a topological space and let
\( \FinalMap \From X \to \ast \) denote the projection to the point.

\begin{definition}[(Constant sheaf)]
	A constant sheaf on \( X \) is a sheaf of the form
	\( \FinalMap^\ast K \) for some \( K \in \Spaces \),
	where \( \Spaces \) is the ∞-category of
	spaces.
\end{definition}

\begin{definition}[(locally constant sheaf)]
	A sheaf \( \Sheaf F \) is said to be locally constant if
	there exists an open covering
	\( \{j_i \From U_i \subset X\}_{i \in I} \) of \( X \)
	such that \( j_i^\ast \Sheaf F \) is constant for each
	\( i \in I \).

	We shall denote by \( \LocallyConstantSheaves(X) \subset
	\Sh(X) \) the full subcategory of locally constant sheaves
	on \( X \)
\end{definition}

\begin{definition}[(Hyperconstant hypersheaf)]
	A hyperconstant hypersheaf on \( X \) is a hypersheaf of the form
	\( \Hyperpullback \FinalMap K \) for some \( K \in \Spaces \).
\end{definition}

\begin{definition}[(Locally hyperconstant hypersheaf)]
	We shall say that an hypersheaf \( \Sheaf H \) is locally
	hyperconstant if 
	there exists an open covering
	\( \{j_i \From U_i \subset X\}_{i \in I} \) of \( X \)
	such that the restriction
	\( j_i^\ast \Sheaf H \)
	is a hyperconstant hypersheaf for each \( i \in I \).

	We shall denote by \( \LocallyHyperConstantHypersheaves(X)
	\subset \Hyp(X) \) the full subcategory of locally hyperconstant
	hypersheaves on \( X \).
\end{definition}

\begin{warning}
	Even though, one has an inclusion
	\( \Hyp(X) \subset \Sh(X) \), one does not have
	\( \LocallyHyperConstantHypersheaves(X) \subset
	\LocallyConstantSheaves(X) \).
\end{warning}

The two main results that we shall use in this article emanate
from Lurie and Haine.
Both authors use a different (but equivalent) definition for
locally constant sheaves
\cite[1.4]{arXiv:2010.06473}
and locally hyperconstant hypersheaves
\cite[A.2.12]{arXiv:0911.0018}:
the one of locally constant sheaves on an ∞-topos.
In what follows, we spend some time explaining why their definition
agrees with the one we have just given.

\begin{definition}[(Locally constant sheaves on an ∞-topos)]
	For an ∞-topos \( \Topos X \),
	let \( \FinalMap \From \Topos X \to \ast \) be a final
	map.
	A constant sheaf on \( \Topos X \), is a sheaf of the form
	\( \FinalMap^\ast K \) for some \( K \in \Spaces \).

	A sheaf \( \Sheaf F \) on \( \Topos X \) is locally constant
	if there is a small family of étale maps
	\( \{j_i \From \Topos U_i \to \Topos X\}_{i \in I} \) such that
	\( \amalg_{i \in I} \Topos U_i \to \Topos X \) be an effective
	epimorphism and \( j_i^\ast \Sheaf F \) be a constant sheaf
	on \( \Topos U_i \) for every \( i \in I \).
\end{definition}

\begin{notation}
	Let us denote by \( \Opens X \) the (nerve of the)
	frame of open subsets
	\( U \subset X \) of
	a topological space \( X \) and let us
	denote by \( \Opens {\Topos X} \), the frame of
	open subtoposes \( \Topos U \subset \Topos X \) of an
	∞-topos \( \Topos X \).
	We shall also denote by \( \Etale {\Topos X} \)
	the ∞-category of étale maps over \( \Topos X \).
\end{notation}

\begin{lemma}
	\label{lemma: opens = left exact localisation of etale}
	The frame of open subtoposes of an ∞-topos \( \Topos X \)
	\[
		\begin{tikzcd}[column sep = huge]
			\Opens {\Topos X}
			\arrow[r, shift right=2, hook]
			&
			\Etale {\Topos X}
			\arrow[l, shift right=2]
		\end{tikzcd}
	\]
	is a left exact and reflexive localisation of
	\( \Etale {\Topos X} \).
\end{lemma}

\begin{proof}
	Open subtoposes correspond to the \( (-1) \)-truncated
	objects of \( \Etale X \), it then is a reflexive subcategory
	of \( \Etale X \)
	\cite[5.5.6.18]{doi:10.1515/9781400830558}.
	Moreover, a morphism \( f \) in \( \Etale {\Topos X} \) becomes
	invertible in \( \Opens {\Topos X} \) if and only if
	\( f \) is an effective epimorphism
	\cite[6.2.3.5(1)]{doi:10.1515/9781400830558}.
	Effective epimorphisms are stable under pullbacks
	\cite[6.2.3.15]{doi:10.1515/9781400830558}
	and thus, the localisation functor preserves finite
	limits
	\cite[6.2.1.1]{doi:10.1515/9781400830558}.
\end{proof}

\begin{lemma}
	\label{lemma: opens}
	Let \( X \) be a topological space, with associated
	∞-topos \( \Topos X \) and hypercomplete subtopos
	\( \Hypercompletion {\Topos X} \subset \Topos X \).

	Then, the maps sending an open \( U \subset X \)
	to the open subtoposes \( \Topos U \subset \Topos X \)
	and
	\( \Hypercompletion {\Topos U}
	\subset \Hypercompletion {\Topos X} \)
	induce equivalences
	\[
		\Opens X
		\IsCanonicallyIsomorphicTo
		\Opens {\Topos X}
		\IsCanonicallyIsomorphicTo
		\Opens {\Hypercompletion {\Topos X}}
	\]
	between
	the frames of open subsets of \( X \),
	open subtoposes of \( \Topos X \) and open
	subtoposes of \( \Hypercompletion {\Topos X} \).
\end{lemma}

\begin{proof}
	For every ∞-topos \( \Topos X \), open subtoposes
	\( \Topos U \subset \Topos X \) correspond
	to subsheaves \( \FinalSheaf_{\Topos U} \) of the terminal
	sheaf \( \FinalSheaf_{\Topos X} \).

	To every open \( U \subset X \) corresponds the characteristic
	sheaf \( \FinalSheaf_U \subset \FinalSheaf_X \)
	on \( X \) where
	\( \FinalSheaf_U(V) \) is punctual whenever \( V \subset U \)
	and is empty otherwise.
	Conversely, let \( \Sheaf F \subset \FinalSheaf_X \) be
	a subsheaf of the terminal sheaf.
	Because \( \Sheaf F \) is a sheaf, if \( \Sheaf F(U) \)
	and \( \Sheaf F(V) \) are both non-empty,
	the value \( \Sheaf F(U \cup V) \) must also be non-empty.
	Thus there is a biggest open subset \( U \subset X \) for
	which \( \Sheaf F(U) \) is non-empty.
	By direct inspection \( \Sheaf F = \FinalSheaf_U \).
	We thus have \( \Opens X \IsCanonicallyIsomorphicTo \Opens
	{\Topos X} \).

	Since hypercompletion \( \Sh(X) \to \Hyp(X) \)
	is a left exact reflexive localisation functor,
	it induces a left exact reflexive localisation
	functor
	\( \Opens {\Topos X}
	\to \Opens {\Hypercompletion {\Topos X}} \).
	Lastly, since the terminal sheaf \( \FinalSheaf_X \)
	is truncated, every subsheaf is also truncated and thus
	hypercomplete
	\cite[6.5.1.14]{doi:10.1515/9781400830558}.
	So, \( \Opens {\Topos X} \IsCanonicallyIsomorphicTo
	\Opens {\Hypercompletion {\Topos X}} \).
\end{proof}

\begin{lemma}
	Let \( X \) be a topological space with
	associated ∞-topos \( \Topos X \).
	Then for every étale map \( \Topos V \to \Topos X \),
	there exists an effective epimorphism
	\[
		\amalg_{i \in I} \Topos U_i \longrightarrow 
		\Topos V
	\]
	over \( \Topos X \) where \( I \) is a small set
	and each \( \Topos U_i \) is an
	open subtopos of \( \Topos X \).
\end{lemma}

\begin{proof}
	By construction of \( \Topos X \), the
	open subtoposes \( \Topos U \subset \Topos X \)
	form a dense
	subcategory of the ∞-category of étale maps over \( \Topos X \).
	Thus given any étale map \( \Topos V \to \Topos X \), the
	canonical map
	\[
		\coprod_{\Topos U \subset \Topos X}
		\coprod_{\substack{\Topos U \to \Topos V\\
		\text{over } \Topos X}}
		\Topos U
		\longrightarrow \Topos V
	\]
	is an effective epimorphism.
	Indeed, this can be checked using injectivity of the pullback
	of subobjects
	\cite[6.2.3.10]{doi:10.1515/9781400830558}.
	Since subobjects of coproducts can be identified with
	products of subobjects
	\cite[6.2.3.9]{doi:10.1515/9781400830558},
	we may start by letting
	let \( \Topos Y, \Topos Z \subset \Topos V \) be two
	open subtoposes of \( \Topos V \) such that
	\( \Topos Y \times_{\Topos V} \Topos U
	= \Topos Z \times_{\Topos V} \Topos U \)
	for every \( \Topos U \to \Topos V \) over \( \Topos X \).
	Then
	\begin{align*}
		\Topos Y
			& \IsCanonicallyIsomorphicTo
			\Topos Y \times_{\Topos V} \Topos V \\
			& \IsCanonicallyIsomorphicTo
			\Topos Y \times_{\Topos V}
			\varinjlim_{
			\substack{\Topos U \to \Topos V \\
			\text{over } \Topos X}} \Topos U
			&& \text{(by density)}
			\\
			& \IsCanonicallyIsomorphicTo
			\varinjlim_{
			\substack{\Topos U \to \Topos V \\
			\text{over } \Topos X}}
			\Topos Y \times_{\Topos V} \Topos U
			&& \text{(by universality of colimits)}
			\\
			& \IsCanonicallyIsomorphicTo
			\varinjlim_{
			\substack{\Topos U \to \Topos V \\
			\text{over } \Topos X}}
			\Topos Z \times_{\Topos V} \Topos U
			&& \text{(by assumption)}
			\\
			& \IsCanonicallyIsomorphicTo
				\Topos Z && \text{(by symmetry)}
	\end{align*}
	one gets \( \Topos Y = \Topos Z \).
\end{proof}

\begin{proposition}[(Translation)]
	Let \( X \) be a topological space, with associated
	∞-topos
	\( \Topos X \) and hypercomplete subtopos
	\( \Hypercompletion {\Topos X} \subset \Topos X \).
	
	Then,
	\[
		\LocallyConstantSheaves(X)
		\IsCanonicallyIsomorphicTo \LocallyConstantSheaves
		(\Topos X)
	\]
	locally constant sheaves on \( X \) are the locally constant
	sheaves on \( \Topos X \).

	Likewise,
	\[
		\LocallyHyperConstantHypersheaves(X)
		\IsCanonicallyIsomorphicTo \LocallyConstantSheaves
		(\Hypercompletion {\Topos X})
	\]
	locally hyperconstant hypersheaves on \( X \) are the
	locally constant sheaves on \( \Hypercompletion {\Topos X} \).
\end{proposition}

\begin{proof}
	Since open embeddings are étale maps, it is clear that
	locally constant sheaves on \( X \) are
	locally constant on \( \Topos X \).

	Let us look at the reverse direction.
	Let \( \Sheaf F \)
	be a locally constant sheaf on \( \Topos X \)
	and let \( \amalg_{i \in I} \Topos V_i
	\twoheadrightarrow \Topos X \) be an
	étale effective epimorphism such that the pullback
	of \( \Sheaf F \) to each \( \Topos V_i \) is constant.

	By the previous lemma,
	for every \( i \in I \) there exists
	a covering
	\( \amalg_{j \in J_i} \Topos U_{ij}
	\twoheadrightarrow \Topos V_i \) over \( \Topos X \),
	by open subtoposes of \( \Topos X \).
	As small coproducts of effective epimorphisms
	are again effective epimorphisms
	\cite[6.2.3.11]{doi:10.1515/9781400830558}
	and the composition of two effective epimorphisms
	is again an epimorphism
	\cite[6.2.3.12]{doi:10.1515/9781400830558},
	we get an effective epimorphism
	\[
		\amalg_{i \in I} \amalg_{j \in J_i} \Topos U_{ij}
		\longrightarrow
		\amalg_{i \in I} \Topos V_i
		\longrightarrow \Topos X
	\]
	covering \( \Topos X \).
	As the pullback of a constant sheaf is again a constant
	sheaf, the restriction of \( \Sheaf F \) to each
	\( \Topos U_{ij} \) is constant.
	Finally, since effective epimorphisms are preserved by
	left exact functors, this one is sent to an effective
	epimorphism
	in \( \Opens {\Topos X} \)
	\UnskipRef{lemma: opens = left exact localisation of etale}
	which is just \( \Opens X \)
	\UnskipRef{lemma: opens},
	that is \( \cup_{i \in I} \cup_{j \in J_i} U_{ij} = X\).

	Now for \( \Hypercompletion {\Topos X} \).
	Since it is a subtopos of \( \Topos X \), every étale
	map with codomain \( \Hypercompletion {\Topos X} \)
	is of the form
	\( \Hypercompletion {\Topos V}
	\to \Hypercompletion {\Topos X} \)
	with \( \Topos V \to \Topos X \) an étale map.
	If
	\( \amalg_{i \in I} \Hypercompletion {\Topos V_i}
	\to \Hypercompletion {\Topos X} \) is an effective epimorphism,
	then as hypercompletion is left exact,
	we get an effective epimorphism
	\[
		\amalg_{i \in I} \amalg_{j \in J_i}
		\Hypercompletion {\Topos U_{ij}}
		\longrightarrow
		\amalg_{i \in I} \Hypercompletion {\Topos V_i}
		\longrightarrow \Hypercompletion {\Topos X}
	\]
	by the same argument as above.
	In addition to the previous arguments, we add that
	for any open \( U \subset X \), constant
	sheaves on \( \Hypercompletion {\Topos U} \) correspond
	to hyperconstant sheaves on \( U \).
\end{proof}

\begin{remark}[(Terminology)]
	\label{Remark: change of terminology}
	Since locally hyperconstant hypersheaves on \( X \)
	are the locally constant sheaves on
	\( \Hypercompletion {\Topos X} \), it has naturally
	led Haine to call them `locally constant hypersheaves'.
	We had to change this terminology in order to distinguish
	between hypersheaves that are locally hyperconstant
	and hypersheaves that are locally constant (as sheaves).
\end{remark}

\subsection{(Hyper)constructible (hyper)sheaves}

Continuing with the two possible functorialities, we shall
obtain constructible sheaves and hyperconstructible hypersheaves.

\begin{definition}[(Constructible sheaves)]
	A sheaf of spaces \( \Sheaf F \)
	on an \( A \)\=/stratified space \( X \)
	is said to be \( A \)-constructible if its
	restriction
	\( s_a^\ast \Sheaf F \)
	to each stratum \( X_a \) is locally constant for every
	\( a \in A \).

	We shall denote by
	\( \ConstructibleSheaves A (X) \subset \Sh(X) \)
	the full subcategory of \( A \)-constructible sheaves on \( X \)
\end{definition}

\begin{remark}
	By construction, given 
	a stratified map \( f \From X \to Y \)
	between an \( A \)-stratified space and a
	\( B \)-stratified space,
	the pullback functor
	\[
		\begin{tikzcd}[ampersand replacement=\&]
			\Sh(Y)
			\arrow[r, "f^\ast"]
				\& \Sh(X) \\
			\ConstructibleSheaves B(Y)
			\ar[u, hook]
			\arrow[r, "f^\ast"]
				\& \ConstructibleSheaves A(X)
				\ar[u, hook]
		\end{tikzcd}
	\]
	preserves constructible sheaves.
\end{remark}

\begin{definition}[(Hyperconstructible hypersheaves)]
	A sheaf \( \Sheaf F \) on \( X \) shall be called
	\( A \)-hyperconstructible
	if the hyperrestriction \( \Hyperpullback {s_a} \Sheaf F \)
	is locally hyperconstant for every \( a \in A \).

	We shall denote
	by \( \HyperConstructibleHypersheaves A (X)
	\subset \Hyp(X) \) the full subcategory of
	\( A \)\=/hyperconstructible hypersheaves.
\end{definition}

\begin{remark}[(Terminology)]
	Following a previous remark
	\UnskipRef{Remark: change of terminology} for the terminology
	about locally constant hypersheaves, what we have chosen
	to call hyperconstructible hypersheaves are the
	constructible hypersheaves of Haine.
	We had to change the terminology in order to distinguish
	between hypersheaves that are hyperconstructible and those
	that are constructible (as sheaves).
\end{remark}

\begin{warning}
	Here again \( \HyperConstructibleHypersheaves A (X)
	\not \subset \ConstructibleSheaves A (X) \), a priori.
\end{warning}

\begin{remark}
	By construction, given a stratified map \( f \From X \to Y \)
	between an \( A \)-stratified space and
	a \( B \)-stratified space, the hyperpullback functor
	\[
		\begin{tikzcd}[ampersand replacement=\&]
			\Hyp(Y)
			\arrow[r, "\Hyperpullback f"]
				\& \Hyp(X) \\
			\HyperConstructibleHypersheaves B(Y)
			\ar[u, hook]
			\arrow[r, "\Hyperpullback f"]
				\& \HyperConstructibleHypersheaves A(X)
				\ar[u, hook]
		\end{tikzcd}
	\]
	preserves hyperconstructible hypersheaves.
\end{remark}

The main distinctive feature of hyperconstructible hypersheaves is their
invariance under stratified homotopy equivalences.
It makes hyperconstructible hypersheaves the natural ∞-analogue
of the usual theory of locally constant sheaves and
constructible sheaves with values in sets.

\begin{theorem}[{\cite[2.3]{arXiv:2010.06473}}]
	\label{theorem: homotopy invariance of hyperconstructibles}
	An \( A \)-stratified homotopy equivalence \( f \From X \to Y \)
	between two \( A \)\=/stratified spaces
	induces an equivalence
	\[
		\begin{tikzcd}
			\HyperConstructibleHypersheaves A (Y)
			\rar["\Hyperpullback f"]
				& \HyperConstructibleHypersheaves A (X)
		\end{tikzcd}
	\]
	between their ∞-categories of
	\( A \)-hyperconstructible hypersheaves.
\end{theorem}

\subsection{Constructible hypersheaves}

We have warned the reader that
one does not have an inclusion
\( \HyperConstructibleHypersheaves A (X) \subset
\ConstructibleSheaves A (X) \)
due to the different functorialities between constructibility and
hyperconstructibility.

They are in fact related by a correspondence
\[
	\begin{tikzcd}
			& \ConstructibleHypersheaves A (X)
			\ar[ld, hook]
			\ar[rd, hook']
				& \\
		\ConstructibleSheaves A (X)
			&
				& \HyperConstructibleHypersheaves A (X)
				\\
			&
	\end{tikzcd}
\]
via the ∞-category of \( A \)-constructible hypersheaves.
This third ∞-category shall become the may object of study in this
article.
The left inclusion is obvious, the right one requires a lemma.

\begin{lemma}
	\label{Lemma: constructible implies hyperconstructible}
	Let \( X \) be a topological space.
	Then
	\[
		\LocallyConstantHypersheaves (X)
		\subset
		\LocallyHyperConstantHypersheaves (X)
	\]
	locally constant hypersheaves are locally hyperconstant.

	More generally, if \( X \) is \( A \)-stratified, then
	\[
		\ConstructibleHypersheaves A (X)
		\subset \HyperConstructibleHypersheaves A (X)
	\]
	\( A \)-constructible hypersheaves on \( X \) are
	\( A \)-hyperconstructible.
\end{lemma}

\begin{proof}
	Assume \( \Sheaf H \) be a locally constant constant hypersheaf.
	Then there exists
	an open covering \( j_i \From U_i \subset X \) such that
	\( j_i^\ast \Sheaf H \) be a constant sheaf.
	But since \( j_i \) is an open embedding,
	\( j_i^\ast \Sheaf H \) is a hypersheaf.
	A sheaf which is both constant and a hypersheaf 
	is hyperconstant.
	constant if and only if 
	By construction, the hypercompletion of a constant sheaf
	is a constant hypersheaf,
	so \( \Hyperpullback {j_i} \Sheaf F \) is a constant hypersheaf
	and \( \Sheaf F \) is hyperlocally constant.

	The case of an \( A \)-constructible sheaf now follows from
	the functoriality of hypercompleted pullbacks
	\UnskipRef{Proposition: hyperpullback}: for every \( a \in A \),
	by assumption
	\( s_a^\ast \Sheaf H \) is locally constant, so let
	\( j_{a, i} \From U_i \subset X_a \) be an open covering so
	that \( j_{a,i}^\ast s_a^\ast \Sheaf F \) is constant.
	Then \( \Hyperpullback {(s_a j_{a, i})} \Sheaf H
	\IsCanonicallyIsomorphicTo
	\Hyperpullback {j_{a,i}} \Hyperpullback {s_a} \Sheaf H \)
	is a constant hypersheaf.
\end{proof}

\begin{remark}
	There is another correspondence relating constructible
	sheaves and hyperconstructible hypersheaves,
	\[
		\begin{tikzcd}
			\ConstructibleSheaves A (X)
			\ar[rd, hook]
				&
					& \HyperConstructibleHypersheaves A (X)
					\ar[dl, hook']
					\\
				&
				\HyperConstructibleSheaves A (X)
				&
		\end{tikzcd}
	\]
	it is the ∞-category of hyperconstructible sheaves.
	But we shall not use it.
\end{remark}

\section{Coincidences}

In this section we shall show that in some general cases, the
∞-categories of hyperconstructible hypersheaves, constructible
hypersheaves and constructible sheaves, actually coincide.

\subsection{Stratum case}

We start with the case of a single stratum.
Lurie has introduced the notion of topological space
‘locally of singular shape’
\cite[A.4.15]{arXiv:0911.0018}.
These are spaces \( X \) for which the counit map
\( \GeometricRealisation {\Sing(U)} \to U \) is a shape
equivalence for every open \( U \subset X \).
Letting \( \FinalMap_X \From X \to \ast \) denote
the projection to the point, if \( X \) is locally of singular shape
then the singular sheaf \( \FinalMap_X^\ast \Sing(X) \) admits a
canonical global
section \( \FinalSheaf_X \to \FinalMap_X^\ast \Sing(X) \).
This canonical section allows the definition of a functor
\[
	\begin{tikzcd}[row sep = tiny]
		\Slice \Spaces {\Sing(X)}
		\rar["\LocallyConstantRealisation_X"]
			& \Sh (X) \\
		K
		\rar[mapsto]
			& \FinalMap_X^\ast K
			\times_{\FinalMap_X^\ast \Sing(X)}
			\FinalSheaf_X
	\end{tikzcd}
\]
which is fully faithful and whose image is equivalent to the
subcategory of locally constant sheaves on \( X \)
\cite[A.2.15]{arXiv:0911.0018}.

When \( X \) is locally of singular shape, the pullback
functor \( \FinalMap_X^\ast \) admits a left adjoint
\( (\FinalMap_X)_! \)
\cite[A.2.8]{arXiv:0911.0018}
and so does \( \LocallyConstantRealisation_X \).
In particular, \( \LocallyConstantRealisation_X \) preserves
all small limits; it also preserves all small colimits.

\begin{proposition}[(Coincidence, stratum case)]
	\label{proposition: coincidence}
	Let \( X \) be a space which is locally of singular shape.
	Then
	\[
		\LocallyHyperConstantHypersheaves(X)
		=
		\LocallyConstantHypersheaves(X)
		=
		\LocallyConstantSheaves(X)
	\]
	locally hyperconstant hypersheaves are locally constant
	and locally constant sheaves are hypersheaves.
\end{proposition}

\begin{proof}
	Since truncation towers converge in
	\( \Slice \Spaces {\Sing(X)} \)
	and \( \LocallyConstantRealisation_X \) commutes
	with limits, all locally constant sheaves on \( X \)
	are hypersheaves
	\cite[A.2.17]{arXiv:0911.0018}.
	Since every open \( U \subset X \) is
	again locally of constant shape, the notions
	of constant sheaves and
	hyperconstant sheaves on \( U \) coincide.
	By ripple effect, locally hyperconstant hypersheaves
	on \( X \)
	are locally constant.
\end{proof}

Being locally of singular shape also gives locally constant
sheaves limits and colimits, they are computed as in
the ∞-category of sheaves.

\begin{lemma}
	\label{Corollary: limits of locally constant}
	Let \( X \) be locally of singular shape.
	The ∞-category \( \LocallyConstantSheaves(X) \)
	admits all small limits and colimits.
	In addition, the inclusion
	\( \LocallyConstantSheaves (X) \subset \Sh(X) \)
	preserves small limits and small colimits.
\end{lemma}

\begin{proof}
	The slice ∞-category \( \Slice \Spaces {\Sing(X)} \)
	has all small limits and colimits and
	\( \LocallyConstantRealisation_X \)
	preserves small limits and small colimits.
\end{proof}

\subsection{Conical case}

Let \( X \) be an \( A \)-stratified space.
The first thing one can ask of \( X \) to have some coincidence
theorem is that each of \( X \) be locally of
singular shape.
This is enough, for example, to guarantee that one can compute
finite limits and small colimits of \( A \)-constructible sheaves on
\( X \).

\begin{lemma}
	\label{Corollary: finite limits of constructibles}
	Let \( X \) be an \( A \)-stratified space.
	Assume that the stratum \( X_a \subset X \) be locally
	of singular shape for each \( a \in A \).
	The ∞-category of \( A \)-constructible sheaves on \( X \)
	admits finite limits and small colimits.
	Moreover, the inclusion
	\( \ConstructibleSheaves A (X) \subset \Sh(X) \)
	preserves finite limits and small colimits.
\end{lemma}

\begin{proof}
	For each \( a \in A \), the functor
	\( s_a^\ast \From \Sh(X) \to \Sh(X_a) \) preserves
	small colimits and finite limits.
	Since \( \LocallyConstantSheaves(X_a) \subset \Sh(X_a) \)
	preserves all small limits and colimits, it follows that
	a finite limit or a small colimit
	of \( A \)-constructible sheaves on \( X \)
	is again \( A \)-constructible. 
\end{proof}

But it is not enough to guarantee any coincidence between the different
notions of constructibility.
Indeed, one needs to add a gluing assumption of the strata together.

The one we shall use here is the conicality introduced by Lurie
\cite[A.5.5]{arXiv:0911.0018};
it is a less demanding condition than most other stratification
hypothesis used in topology.
A space is conically stratified when each point admits a
coordinate decomposition with on one side, a local coordinate
dependant on the stratum and on the other side a radial coordinate
describing a neighbourhood of the point around the stratum.

\begin{definition}[(Open cone)]
	For a topological space \( X \), the \emph{open cone}
	of \( X \) is the set 
		\[
			\Cone(X)
			\coloneqq \{ 0 \} \amalg (\Reals_+^\ast \times X)
		\]
	with topology defined as follows:
	A subset \( U \subset \Cone(X) \)
	is open if and only if
	\( U \cap (\Reals_+^\ast \times X) \) is open,
	and if \( 0 \in U \), then
	\( (0, \varepsilon) \times X \subset U \)
	for some positive real number \( \varepsilon \). 

	If \( X \) is stratified over a poset \( A \), then
	\( \Cone(X) \) is naturally stratified over
	the poset \( A^\triangleleft \) obtained from \( A \) by
	adding a new element smaller than every other element of
	\( A \).
\end{definition}

\begin{warning}
	One should not confuse the \emph{cone}
	on \( X \) with the \emph{collapsed rectangle}
	defined as the quotient
	\( \Reals_+ \times X / \{0\} \times X \).
	When \( X \) is compact and separated, the cone on
	\( X \) and the collapsed rectangle on \( X \) are
	homeomorphic. This is no longer true in the general
	case: the cone on the open interval
	\( (0,1) \) can be embedded in \( \Reals^2 \), whereas
	the collapsed rectangle on \( (0,1) \) is not
	metrizable.

	If \( (X,d) \) is a metric space, the topology of the cone
	\( \Cone(X) \) is metrizable by letting
	\( d((\lambda, x),(\mu, y) = \max(|\lambda - \mu|, d(x,y)) \)
	and by adding \( d(0, (\lambda, x)) = \lambda \).
\end{warning}

\begin{definition}[(Conically stratified space)]
	Let \( f \From X \to A \) be a stratified topological space. 
	Let \( a \in A \) and let
	\( U \subset X_a \) be an open subset.
	We shall say that an open \( V \subset X \)
	is a conical extention of \( U \) if there exists
	a stratified
	space \( L \) over \( A_{a <} \) such that
	\( V \) is homeomorphic
	to \( U \times \Cone(L) \)
	over the poset map
	\( A_{a<}^\triangleleft \IsCanonicallyIsomorphicTo
	A_{a \leq} \subset A\).
	We shall say that \( X \) is conically \( A \)-stratified
	if for every \( a \in A \) every point \( x \in X_a \)
	admits an open neighbourhood in \( X_a \)
	that can be conically extended to \( X \).
\end{definition}

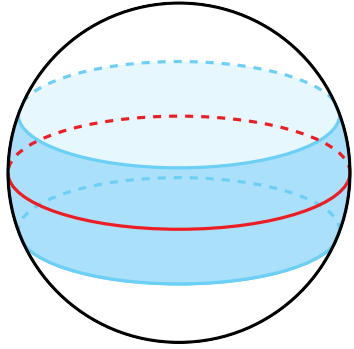
\begin{figure}
	\usetikzlibrary{intersections, pgfplots.fillbetween}
	\begin{tikzpicture}[scale = 1.5]
		\pgfsetlayers{pre main,main}
	\newcommand\LatitudeFront[2]{%
	\draw[very thick, color = #2, name path = #1F]
		(#1:1.5) arc (0:-180:{1.5*cos(#1)} and {0.5*cos(#1)});
	}
	\newcommand\LatitudeBack[2]{%
	\draw[dashed, very thick, color = #2, name path = #1B]
		(#1:1.5) arc (0:180:{1.5*cos(#1)} and {0.5*cos(#1)});
	}
	\LatitudeFront{-20}{Cyan!50};
	\LatitudeBack{-20}{Cyan!50};
	\foreach \i in {1,...,39}
	{\LatitudeFront{\i/2}{Cyan!30};
	\LatitudeFront{-\i/2}{Cyan!30};};
	\LatitudeFront{20}{Cyan!50};
	\LatitudeBack{20}{Cyan!50};
	\LatitudeFront{-20}{Cyan!50};
	\LatitudeBack{-20}{Cyan!50};
	\tikzfillbetween[of=20F and 20B]{Cyan, opacity=0.1};
	\LatitudeFront{0}{Red};
	\LatitudeBack{0}{Red};
		\draw[very thick] (0,0) circle (1.5 cm);
		\draw[very thick, color = white] (0,0) circle (1.53 cm);
	\end{tikzpicture}
	\caption{A submanifold \( \textcolor{Red}{\bm N}
		\subset \bm M \)
		with a tubular neighbourhood
		is an example of
	a conically stratified space with two strata.}
\end{figure}

\begin{argument}[(Reduction)]
	\label{Argument: local structure of cones}
	Let us gather some properties of conically stratified
	spaces that shall be used as a core reduction argument
	in the proofs.
	\begin{itemize}
		\item
			If a conically \( A \)-stratified space
			\( X \) is paracompact,
			then each stratum \( X_a \) can be
			covered with opens \( U \)
			admitting a paracompact conical extension
			\cite[A.5.16]{arXiv:0911.0018}.
			As a consequence, for any local problem
			on paracompact \( X \), one can
			assume that \( X = X_a \times \Cone(L) \);
		\item
			In a paracompact space
			\( X_a \times \Cone(L) \), the closed
			subspace \( X_a \subset X_a \times \Cone(L) \)
			is also paracompact, thus
			\( \FSigma \) open subsets \( W \subset X_a \)
			form a basis of paracompact open
			subsets, stable under intersection
			\cite[7.1.1.1]{doi:10.1515/9781400830558}.
			Moreover, if \( W \subset X_a \) is an
			open \( \FSigma \), then
			\( W \times \Cone(L)
			\subset X_a \times \Cone(L) \)
			is again an open \( \FSigma \) of
			\( X_a \times \Cone(L) \) and is
			again paracompact;
		\item
			For a paracompact \( W \), the
			closed subspace
			\( W \subset W \times \Cone(L) \)
			admits a basis of open neighbourhoods
			\( W \subset V \subset W \times \Cone(L) \)
			all of which are homeomorphic to
			\( W \times \Cone(L) \) as stratified
			spaces \cite[A.5.12]{arXiv:0911.0018};
		\item
			Combining the above arguments, every point
			\( x \in X \)
			admits a basis of conical open neighbourhoods
			\( V_x
			\IsIsomorphicTo U_x \times \Cone(L) \).
	\end{itemize}
\end{argument}
\begin{proposition}[(Coincidence, conical case)]
	\label{Corollary: hyperconstructible = constructible}
	Let \( X \) be a paracompact \( A \)-stratified space,
	such that each stratum \( X_a \) be
	locally of singular shape, for \( a \in A \).
	Then,
	for every \( A \)-hyperconstructible sheaf \( \Sheaf H \)
	on \( X \), the canonical map
	\( s_a^\ast \Sheaf H \to \Hyperpullback{s_a} \Sheaf H \)
	is an equivalence
	on \( X_a \), for each \( a \in A \) and thus
	\[
		\HyperConstructibleHypersheaves A (X)
		= \ConstructibleHypersheaves A (X)
	\]
	\( A \)-hyperconstructible hypersheaves on \( X \) are
	\( A \)-constructible.
\end{proposition}

\begin{proof}
	Constructible hypersheaves are
	hyperconstructible
	\UnskipRef{Lemma: constructible implies hyperconstructible}.
	Also, since each stratum \( X_a \) is locally of singular
	shape, for \( a \in A \), locally hyperconstant
	hypersheaves on \( X_a \) coincide with
	locally constant sheaves on \( X_a \)
	\UnskipRef{proposition: coincidence}.
	It shall then be enough to show that
	for every \( A \)-hyperconstructible hypersheaf \( \Sheaf H \)
	the canonical map
	\( \alpha \From s_a^\ast \Sheaf H
	\to \Hyperpullback {s_a} \Sheaf H \)
	is an equivalence on \( X_a \).

	This question is local on \( X_a \),
	so by the reduction arguments for conically stratified space
	\UnskipRef{Argument: local structure of cones},
	we can reduce to the case were \( X = X_a \times \Cone(L) \).
	Continuing the reduction, it is enough to show that
	\( \alpha(W) \) is an equivalence on each
	\( \FSigma \) open subset \( W \subset X_a \) since this
	is a basis stable under intersection.
	For each such \( W \), \( W \times \Cone(L) \) is again
	paracompact and
	we can thus reduce to show that \( \alpha(W) \) is an
	equivalence in the case where \( X = W \times \Cone(L) \).

	Now because \( X \) is paracompact,
	\( s_a^\ast \Sheaf H (W) \IsCanonicallyIsomorphicTo
	\varinjlim_{W \subset V} \Sheaf H(V) \)
	\cite[7.1.5.6]{doi:10.1515/9781400830558}.
	Because \( W \) is paracompact, the neighbourhoods
	\( V \) can be taken homeomorphic to \( W \times \Cone(L) \)
	as stratified spaces
	\UnskipRef{Argument: local structure of cones}.
	In such a case, by homotopy invariance
	\UnskipRef{theorem: homotopy invariance of hyperconstructibles},
	the restriction map
	\( \Sheaf H(V) \to
	\Hyperpullback {s_a} \Sheaf H(W) \) is an equivalence,
	from which we finally
	get that \( \alpha(W) \) is an equivalence.
\end{proof}

\subsection{\( \DD \)-spaces}

We have seen that on conically \( A \)-stratified spaces,
hyperconstructibility
coincides with constructibility for hypersheaves.
In order to add coincidence with constructible sheaves,
as in the stratum case, one needs
to add an assumption on the poset allowing induction on depth.
Namely, one needs to assume that the poset satisfy the ascending
chain condition.

Adding the ascending chain condition to the poset \( A \), we get
a type of spaces that shall become the building brick of the next
construction; we thus give it a name.

\begin{definition}[(\(\DD\)-space)]
	A \( \DD \)-space (for ‘good’ depth
	stratified space) is an \( A \)-stratified space \( X \)
	such that
	\begin{itemize}
		\item
			\( X \) is paracompact;
		\item
			the stratum \( X_a \subset X \)
			is locally of singular shape,
			for each \( a \in A \);
		\item
			\( X \) is conically \( A \)-stratified;
		\item
			\( A \) satisfies the ascending chain condition.
	\end{itemize}
\end{definition}

\begin{remark}
	Any \( \mathrm{C}^0 \)-stratified space in the sense of
	Ayala-Francis-Tannaka is a \( \DD \)\=/space
	\cite[2.1.15]{doi:10.1016/j.aim.2016.11.032}.
\end{remark}

\begin{proposition}[(Coincidence, \( \DD \)-space case)]
	\label{Proposition: D-spaces}
	Let \( X \to A \) be a \( \DD \)-space, then
	\[
		\HyperConstructibleHypersheaves A (X)
		\IsCanonicallyIsomorphicTo
		\ConstructibleHypersheaves A (X)
		\IsCanonicallyIsomorphicTo
		\ConstructibleSheaves A (X)
	\]
	\( A \)-constructible sheaves on \( X \) are hypersheaves
	and \( A \)-hyperconstructible hypersheaves are
	\( A \)-constructible.
\end{proposition}

\begin{proof}
	The first equality follows from the conical
	case
	\UnskipRef{Corollary: hyperconstructible = constructible}.
	The second can be proven by induction on the depth of
	\( A \) {\cite[A.5.9]{arXiv:0911.0018}}.
\end{proof}

\subsection{Colimit and conical \( \DDO \)-spaces}

We now turn to the main object of study: a class of stratified
spaces on which to extend the representation theorem.
The idea is to consider to those posets \( A \) which
do not satisfy the ascending chain condition but which
can be obtained as a countable union of closed subposets
satisfying the ascending chain condition.

\begin{definition}[(\(\DDO\)-space)]
	A \( \DDO \)-space is a stratified space \( X \to A \)
	such that
	\begin{itemize}
		\item
			\( X \) is paracompact;
		\item
			\( A \) is \( \OrdinalOmega \)-stratified;
		\item
			\( X_{\leq n} \to A_{\leq n} \) is
			a \( \DD \)-space
			for every \( n < \OrdinalOmega \).
	\end{itemize}
	We shall say that
	\( X \to A \to \OrdinalOmega \) is a
	\begin{description}
		\item[conical \( \DDO \)-space]
			if \( X \)
			is conically \( A \)-stratified;
		\item[colimit \( \DDO \)-space]
			if \( X \) coincides with
			the colimit
			\( X_{< \OrdinalOmega} \coloneqq
			\varinjlim_{n < \OrdinalOmega}
			X_{\leq n} \).
	\end{description}
\end{definition}

One may legitimately ask: why divide \( \DDO \)-spaces into
two categories?
This is because
the very topology of the cone is often incompatible with a colimit
topology.
For example, if \( L \) is an \( \OrdinalOmega \)-stratified, then
there is a continuous bijection
\[
	\textstyle
	\varinjlim_{n < \OrdinalOmega} \Cone(L_{\leq n})
	\longrightarrow \Cone(L)
\]
which is not a homeomorphism in the general case.
This leads to an impossibility theorem, where the two conditions
become mutually exclusive.

\begin{remark}
	\label{Remark: converging sequence in colimit topology}
	Let \( X_0 \hookrightarrow
	\cdots \hookrightarrow X_p \hookrightarrow \cdots \)
	be a sequence of closed embeddings
	between \( \TOne \) topological
	spaces and let \( X \) denote its colimit.
	Then every morphism \( K \to X \) with \( K \) compact
	factors through one \( X_p \subset X \)
	\cite[2.4.2]{isbn:9780821843611}.

	As a consequence, a sequence
	\( (x_n)_{n \in \Naturals} \) in \( X \),
	converges only if it is bounded.
\end{remark}

\begin{theorem}[(Impossibility)]
	\label{Theorem: impossibility}
	Assume \( X \) be an \( A \)-stratified space
	such that
	\begin{itemize}
		\item
			\( X \) be \( \TOne \)
			and not empty;
		\item
			\( X_{\leq a} \subset X \) have empty interior
			for each \( a \in A \);
		\item
			\( X \IsCanonicallyIsomorphicTo
			\varinjlim_{a \in A}
			X_{\leq a} \);
		\item
			\( A \) contain an ascending
			chain,
	\end{itemize}
	then \( X \) is not conically stratified.
\end{theorem}

\begin{proof}
	One can assume that \( A = \OrdinalOmega \)
	and that \( X_0 \) is not empty without loss
	of generality.
	Let \( x \in X_{\leq 0} \), if \( X \) is conically stratified,
	then there exists \( Z \) and \( Y \) such that
	\( Z \times \Cone(Y) \) is stratifiedly homeomorphic to
	an open neighbourhood \( U_x \) of \( x \).
	Since each \( X_{\leq n} \) has empty interior, \( U_x \)
	is not contained in any of them.
	One can thus find a sequence in \( U_x \)
	whose image in \( \OrdinalOmega \) is strictly increasing.
	Let \( (y_n) \) be the corresponding coordinate
	in \( Y \) of this sequence.
	Then for any \( z \in Z \), the sequence
	\( (z, (\lambda_n, y_n)) \) converges to \( (z,0) \) in
	\( Z \times \Cone(Y) \) for any sequence \( \lambda_n \to 0 \)
	but cannot converge
	in \( X \) since it is not bounded in the stratification
	\UnskipRef{Remark: converging sequence in colimit topology}.
\end{proof}

Nevertheless, conical and colimit \( \DDO \)-spaces are very close;
one can always change the global topology of a conical \( \DDO \)-space
to make it become a colimit \( \DDO \)-space.

\begin{lemma}
	\label{Lemma: colimit endofunctor}
	Let \( X \to A \to \OrdinalOmega \) be a \( \DDO \)-space.
	Then \( X_{< \OrdinalOmega} \to A \to \OrdinalOmega \)
	is a colimit \( \DDO \)-space.
	Moreover, if \( Y \to X \) is a continuous map
	making \( Y \to X \to A \to \OrdinalOmega \)
	into a \( \DDO \)-space, then one gets a continuous map
	\( Y_{< \OrdinalOmega} \to X_{\OrdinalOmega} \) over \( A \).
\end{lemma}

\begin{proof}
	The only non-trivial thing to check is the paracompactness.
	But since by hypothesis each \( X_{\leq n} \) is paracompact,
	the space \( X_{< \OrdinalOmega} \) is built as a sequential
	colimit of closed embeddings of paracompact spaces and
	is thus paracompact
	\cite[8.2]{doi:10.2307/1969615}.
\end{proof}

\begin{lemma}
	\label{Lemma: push of constructible hypersheaf}
	Let \( X \to A \to \OrdinalOmega \) be a \( \DDO \)-space.
	For each \( n < \OrdinalOmega \), let us denote by \( j(n) \)
	the closed stratified embedding \( X_{\leq n} \subset X \).
	Then for each \( A_{\leq n} \)-constructible sheaf
	\( \Sheaf F \) on \( X_{\leq n} \),
	\[
		\Sheaf F \in \ConstructibleSheaves {A_{\leq n}}
		(X_{\leq n})
		\implies j(n)_\ast \Sheaf F
		\in 
		\ConstructibleHypersheaves A (X)
	\]
	its pushforward \( j(n)_\ast \Sheaf F \)
	is an \( A \)-constructible hypersheaf on \( X \).
\end{lemma}

\begin{proof}
	Let \( a \in A \).
	Either \( X_a \cap X_{\leq n} = \emptyset \), in
	which case
	\( s_a^\ast j(n)_\ast \Sheaf F \) is the initial sheaf
	by proper base change
	\cite[7.3.2.13]{doi:10.1515/9781400830558},
	or \( X_a \subset X_{\leq n} \), in which case one has
	\( s_a^\ast j(n)_\ast \Sheaf F \IsCanonicallyIsomorphicTo
	s_a^\ast \Sheaf F \) again by proper base change and is
	locally constant by hypothesis.

	Since by assumption \( X_{\leq n} \) is a \( \DD \)-space,
	\( \Sheaf F \) is a hypersheaf
	\UnskipRef{Proposition: D-spaces}
	and as pushforwards preserve
	hypersheaves, \( j(n)_\ast \Sheaf F \) is also a hypersheaf.
\end{proof}

\begin{theorem}[(Dévissage, colimit case)]
	\label{Theorem: dévissage colimit}
	Let \( X \to A \to \OrdinalOmega \)
	be a colimit \( \DDO \)-space
	The inclusion maps \( j(n) \From X_{\leq n} \subset X \)
	induce an adjunction
	\[
		\begin{tikzcd}[column sep = huge]
			\Sh(X) \arrow[r, shift left=2,"j^\ast"]
			&
			\varprojlim_{n < \OrdinalOmega}
			\Sh(X_{\leq n})
			\arrow[l, shift left=2, "j_\ast"]
		\end{tikzcd}
	\]
	which is an equivalence of ∞-categories and which
	reduces to an equivalence
	\[
		\textstyle
		\ConstructibleSheaves A (X)
		\IsCanonicallyIsomorphicTo 
		\varprojlim_{n < \OrdinalOmega}
		\ConstructibleSheaves {A_{\leq n}} (X_{\leq n})
	\]
	between the ∞-category of \( A \)-constructible sheaves
	on \( X \) and the inverse limit of the ∞-categories
	of constructible sheaves on each \( X_{\leq n} \).
	Moreover,
	\[
		\ConstructibleHypersheaves A (X)
		=
		\ConstructibleSheaves A (X)
	\]
	\( A \)-constructible sheaves on \( X \) are hypersheaves.
\end{theorem}

\begin{proof}
	The fact that \( j^\ast \IsAdjointTo j_\ast \) is an
	equivalence of ∞-categories follows from the fact that
	\( X \) is the colimit of a sequence of closed embeddings
	of paracompact spaces
	\cite[7.1.5.8]{doi:10.1515/9781400830558}.
	Then the restriction to the subcategories of contructible
	sheaves follow directly.
	Finally, since
	since \( X_{\leq n} \) is a \( \DD \)-space for every
	\( n < \OrdinalOmega \), all \( A_{\leq n} \)-constructible
	sheaves on \( X_{\leq n} \) are hypersheaves
	\UnskipRef{Proposition: D-spaces} and as
	limits of hypersheaves are again hypersheaves, we get that
	all \( A \)-constructible sheaves on \( X \)
	are hypersheaves.
\end{proof}

This theorem for colimit \( \DDO \)-spaces,
together with the coincidence proposition for conically stratified
spaces shall let us see that constructible hypersheaves (which are
a priori not functorial) inherit the functoriality of
hyperconstructible hypersheaves.
In particular it shall also inherit its homotopy invariance.

\begin{corollary}[(Functoriality of constructible hypersheaves)]
	Let \( X \to A \to \OrdinalOmega \) be a colimit or a conical
	\( \DDO \)-space
	and let \( Y \to B \) be a stratified space.
	Let \( f \From X \to Y \) be a stratified map.
	Then for any \( B \)-constructible hypersheaf
	\( \Sheaf H \)
	\[
		\Sheaf H \in \ConstructibleHypersheaves B (Y)
		\implies \Hyperpullback{f} \Sheaf H
		\in 
		\ConstructibleHypersheaves A (X)
	\]
	its hyperpullback \( \Hyperpullback{f} \Sheaf H \) is
	\( A \)-constructible.

	In particular, if \( f \) is a stratified homotopy
	equivalence between two conical or colimit \( \DDO \)-spaces,
	then \( \Hyperpullback{f} \) induces an equivalence
	between the ∞-categories of constructible hypersheaves.
\end{corollary}

\begin{proof}
	When \( X \) is a conical \( \DDO \)-space: then
	since \( \Sheaf H \) is a \( B \)-constructible
	hypersheaf, it is \( B \)-hyperconstructible
	\UnskipRef{Lemma: constructible implies hyperconstructible}.
	The hyperpullback \( \Hyperpullback{f} \Sheaf H \)
	is then \( A \)-hyperconstructible.
	But since \( X \) is conically stratified,
	it is also \( A \)\=/constructible by coincidence
	\UnskipRef{Corollary: hyperconstructible = constructible}.

	When \( X \) is a colimit \( \DDO \)-space, then
	\( f^\ast \Sheaf H \) is \( A \)-constructible because
	\( \Sheaf H \) is \( B \)-constructible.
	By the previous theorem \( f^\ast \Sheaf H \) is then
	a hypersheaf and so, the canonical map
	\( f^\ast \Sheaf H \to \Hyperpullback{f} \Sheaf H \)
	is an equivalence of sheaves on \( X \).

	As a consequence, when restricted to colimit and
	conical \( \DDO \)-spaces, constructible hypersheaves form
	a subfunctor of the functor of hyperconstructible hypersheaves
	\UnskipRef{Corollary: hyperconstructible = constructible}.
	It is thus invariant under stratified homotopy equivalences
	\UnskipRef{theorem: homotopy invariance of hyperconstructibles}.
\end{proof}

\begin{theorem}[(Dévissage, conical case)]
	\label{Theorem: dévissage conical}
	Let \( X \to A \to \OrdinalOmega \) be a \( \DDO \)-space
	and assume that \( X \) is either a conical
	or a colimit \( \DDO \)-space.
	The inclusion maps \( j(n) \From X_{\leq n} \subset X \)
	induce an adjunction
	\[
		\begin{tikzcd}[column sep = huge]
			\Sh(X) \arrow[r, shift left=2,"j^\ast"]
			&
			\varprojlim_{n < \OrdinalOmega}
			\Sh(X_{\leq n})
			\arrow[l, shift left=2, "j_\ast"]
		\end{tikzcd}
	\]
	which reduces to an equivalence
	\[
		\textstyle
		\ConstructibleHypersheaves A (X)
		\IsCanonicallyIsomorphicTo 
		\varprojlim_{n < \OrdinalOmega}
		\ConstructibleSheaves {A_{\leq n}} (X_{\leq n})
	\]
	between the ∞-category of \( A \)-constructible hypersheaves
	on \( X \) and the inverse limit of the ∞-categories
	of constructible sheaves on each \( X_{\leq n} \).
\end{theorem}

\begin{proof}
	Using the dévissage theorem for the colimit case,
	we shall
	identify the right hand side of the equivalence with the
	∞-category of \( A \)-constructible sheaves
	on \( X_{< \OrdinalOmega} \)
	and see the adjunction \( j^\ast \IsAdjointTo j_\ast \)
	as steming from the canonical map \( j \From X_{<\OrdinalOmega}
	\to X \).
	For a sheaf \( \Sheaf F \) on \( X_{< \OrdinalOmega} \), we
	shall denote by \( \Sheaf F_{\leq n} \) its restriction to
	\( X_{\leq n} \) for each \( n < \OrdinalOmega \)
	and by \( \Sheaf F_a \) its restriction to \( X_a \) for
	each \( a \in A \).

	We shall start by showing that for every \( A \)-constructible
	sheaf \( \Sheaf F \) on \( X_{< \OrdinalOmega} \) the map
	\( \psi_{\Sheaf F} \From
	s_a^\ast j_\ast \Sheaf F \to \Sheaf F_a \)
	is an equivalence for each \( a \in A \).
	This is of local nature, thus
	by the reduction arguments
	\UnskipRef{Argument: local structure of cones}, one can
	assume that \( X = X_a \times \Cone(L) \).
	Continuing the reduction, it shall then suffice to show that
	\( \psi(W) \) is an equivalence for every
	\( \FSigma \) open subset \( W \subset X_a \),
	since this is a basis stable under intersection.
	For each such \( W \), \( W \times \Cone(L) \) is again
	paracompact and we thus reduce to show that
	\( \psi(X_a) \) is an equivalence.

	We shall prove by induction on \( k < \OrdinalOmega \) that
	for every \( \Sheaf F \)
	the map \( \psi_{\Sheaf F}(X_a) \) is \( k \)-connective.
	For the case \( k = 0 \), the canonical map
	\[
		\textstyle
		\PiZero j_\ast \Sheaf F(X)
		\IsCanonicallyIsomorphicTo
		\PiZero\left(\varprojlim_{n < \OrdinalOmega}
		j(n)_\ast \Sheaf F_{\leq n}(X)\right)
		\longrightarrow
		\varprojlim_{n < \OrdinalOmega}
		\PiZero j(n)_\ast \Sheaf F_{\leq n}(X)
	\]
	is surjective.
	As each \( j(n)_\ast \Sheaf F_{\leq n} \) is a constructible
	hypersheaf
	\UnskipRef{Lemma: push of constructible hypersheaf},
	by homotopy invariance
	\UnskipRef{theorem: homotopy invariance of hyperconstructibles}
	the stratified deformation
	retract \( X_a \times \Cone(L) \to X_a \) gives us
	\( j(n)_\ast \Sheaf F_{\leq n}(X) \IsCanonicallyIsomorphicTo
	\Sheaf F_a(X_a) \).
	We deduce that \( \PiZero j_\ast \Sheaf F(X) \to
	\PiZero \Sheaf F_a (X_a) \) is surjective, and thus that
	\( \PiZero \psi_{\Sheaf F}(X_a) \),
	through which the surjection factors, is surjective.

	Assume that we have shown that \( \psi_{\Sheaf F}(X_a) \)
	is \( k \)-connective
	for every \( A \)\=/constructible
	\( \Sheaf F \), for some \( k < \OrdinalOmega \).
	To show that \( \psi_{\Sheaf F}(X_a) \)
	is \( (k+1) \)-connective, it is
	equivalent to show that for every section
	\( \eta \in s_a^\ast j_\ast \Sheaf F(X_a) \), the
	induced map
	\[
		\begin{tikzcd}
			\ast
			\times_{s_a^\ast j_\ast \Sheaf F(X_a)}
			\ast
			\rar["\psi'"]
				& \ast \times_{\Sheaf F_a(X_a)}
				\ast
		\end{tikzcd}
	\]
	is \( k \)-connective.
	Because \( X \) is paracompact and \( X_a \subset X \)
	is a closed embedding, there exists an open neighbourhood
	\( V \) of \( X_a \) in \( X \) on which \( \eta \) can
	be extend to a section \( \overline \eta \)
	\cite[7.1.5.5]{doi:10.1515/9781400830558}.
	Since \( X_a \) is paracompact, shrinking \( V \) if
	necessary, one can assume that \( V \) is homeomorphic to
	\( X_a \times \Cone(L) \) as a stratified space
	\UnskipRef{Argument: local structure of cones}.
	We may thus reduce to the case where \( \overline \eta \)
	is a global section of \( j_\ast \Sheaf F \) on \( X \).
	The induced map \( \FinalSheaf_{X_{< \OrdinalOmega}} \to
	\Sheaf F \) allows us to define
	\( \Sheaf G \coloneqq \FinalSheaf \times_{\Sheaf F}
	\FinalSheaf \).
	The sheaf \( \Sheaf G \) is again \( A \)-constructible
	since it is a finite limit of \( A \)-constructible sheaves
	\UnskipRef{Corollary: finite limits of constructibles}.
	By left exactness,
	one has \( \psi' = \psi_{\Sheaf G}(X_a) \) which is then
	\( k \)-connective by the induction hypothesis.

	We deduce that
	\( j_\ast \Sheaf F \) is \( A \)-constructible.
	Since hypersheaves are stable under pushforwards and
	all \( A \)-constructible sheaves on \( X_{<\OrdinalOmega} \)
	are hypersheaves, \( j_\ast \Sheaf F \)
	is an \( A \)-constructible hypersheaf.

	So, \( j^\ast j_\ast \Sheaf F \) is also \( A \)-constructible 
	and thus, a hypersheaf.
	We also deduce that for every point
	\( x \in X_{<\OrdinalOmega} \),
	the counit map \( (j^\ast j_\ast \Sheaf F)_x \to \Sheaf F_x \)
	is an equivalence and thus that
	\( j^\ast j_\ast \Sheaf F \to \Sheaf F \) is an equivalence.

	We now show that for every \( A \)-constructible
	hypersheaf \( \Sheaf H \) on \( X \), the
	unit map
	\( \upsilon \From \Sheaf H \to j_\ast j^\ast \Sheaf H \) is
	an equivalence.
	Since \( \Sheaf H \) is a hypersheaf, one needs to show
	that
	\( \upsilon_x \From \Sheaf H_x
	\to (j_\ast j^\ast \Sheaf H)_x \) is an equivalence
	for every \( x \in X \).
	For this it is enough to show that
	\( \upsilon(V_x) \From \Sheaf H(V_x)
	\to j_\ast j^\ast \Sheaf H (V_x) \)
	is an equivalence for a basis of open subsets
	\( x \in V_x \subset X \).
	By the conical nature of \( X \), it is possible to
	select \( V_x \IsIsomorphicTo U_x \times \Cone(L) \)
	\UnskipRef{Argument: local structure of cones}.
	We then reduce to showing that
	\( \upsilon(X) \) is an equivalence
	in the case where \( X = X_a \times \Cone(L) \).

	Let \( p \) be the unique natural number such that
	\( X_a \subset X_p \).
	Since
	\[
		j_\ast j^\ast \Sheaf H (X_a \times \Cone(L))
		\IsCanonicallyIsomorphicTo
		\textstyle
		\varprojlim_{p \leq n < \OrdinalOmega} \Sheaf H_{\leq n}
		(X_a \times \Cone(L_{\leq n}))
	\]
	it shall be enough to see that
	\( \Sheaf H (X_a \times \Cone(L)) \to
	\Sheaf H_{\leq n}(X_a \times \Cone(L_{\leq n})) \)
	is an equivalence for every \( n \geq p \).
	One has a commutative diagram of restriction maps
	\[
		\begin{tikzcd}
			\Sheaf H (X_a \times \Cone(L))
			\ar[rr]
			\ar[rd]
			&& \Sheaf
			H_{\leq n}(X_a \times \Cone(L_{\leq n}))
			\ar[ld]
			\\
			&
			s_a^\ast\Sheaf H(X_a)
			\IsCanonicallyIsomorphicTo
			\Hyperpullback{s_a}\Sheaf H(X_a)
			\UnskipRef{Corollary:
			hyperconstructible = constructible}
			&
		\end{tikzcd}
	\]
	for which both of the vertical arrows are equivalences by
	homotopy invariance
	\UnskipRef{theorem: homotopy invariance of hyperconstructibles},
	We conclude using the two-out-of-three property of equivalences.
\end{proof}

\begin{remark}
	Combining the two dévissage theorems with the coincidence
	theorem in the conical case
	\UnskipRef{Corollary: hyperconstructible = constructible},
	we see that
	for a conically \( A \)-stratified \( \DDO \)\=/space,
	one has the following coincidences
	\[
		\HyperConstructibleHypersheaves A (X)
		\IsCanonicallyIsomorphicTo
		\ConstructibleHypersheaves A (X)
		\IsCanonicallyIsomorphicTo
		\ConstructibleHypersheaves A (X_{< \OrdinalOmega})
		\IsCanonicallyIsomorphicTo
		\ConstructibleSheaves A (X_{<\OrdinalOmega})
	\]
	which can be extended with
	\[
		\HyperConstructibleHypersheaves A (X_{< \OrdinalOmega})
		\IsCanonicallyIsomorphicTo
		\ConstructibleHypersheaves A (X_{< \OrdinalOmega})
	\]
	a coincidence between
	\( A \)\=/hyperconstructible hypersheaves
	on \( X_{< \OrdinalOmega} \)
	and \( A \)\=/constructible hypersheaves on \( X \).
	Contrarily to the conical case, this last equality might
	not happen systematically for colimit \( \DDO \)-spaces.
	However, one can show that this is the case whenever
	the colimit \( \DDO \)-space arises from a
	conical \( \DDO \)-space.
\end{remark}

\section{Representation via exit paths}

In this section, we extend the representation theorem
for constructible sheaves on \( \DD \)-spaces given by Lurie, to
a representation theorem for constructible hypersheaves on
colimit or conical \( \DDO \)-spaces.

\subsection{Exit paths ∞-category}

\begin{definition}
	Let \( p < \OrdinalOmega \), we shall view the topological
	simplex
	\[
		 \GeometricRealisation {\Simplex p}
		 = \left\{ (t_0, \dots, t_p) \in [0,1]^{p+1}
		 : t_0+ \dots + t_p = 1 \right\}
	\]
	as a stratified space over the poset
	\( \{0 < \dots < p\} \) with
	\( \GeometricRealisation {\Simplex p}_{\leq i} \)
	being the set of tuples \( (t_0, \dots, t_i, 0, \dots, 0) \)
	for every \( i \leq p \).

	Given a stratified space \( X \to A \), we let
	\( \Exit A (X) \) denote
	the simplicial set whose \( p \)-simplicies
	are the stratified maps \( \GeometricRealisation {\Simplex p}
	\to X \), for every \( p < \OrdinalOmega \).
\end{definition}

\begin{theorem}[{\cite[A.6.4]{arXiv:0911.0018}}]
	Let X be a conically \( A \)-stratified space.
	Then \( \Exit A (X) \) is an ∞\=/category.
\end{theorem}

\begin{proposition}
	\label{Proposition: Exit of DDO-space}
	Let \( X \to A \) be a stratified space.
	Assume \( A \) be also
	stratified over a filtered poset
	\( \Lambda \) and assume that
	\( X_{\leq \lambda} \) be
	conically \( A_{\leq \lambda} \)-stratified space
	for each \( \lambda \in \Lambda \).
	Then, the simplicial set \( \Exit A (X) \) is an ∞-category and
	the canonical isomorphism of simplicial sets
	\[
		\textstyle
		\varinjlim_{\lambda \in \Lambda}
		\Exit {A_{\leq \lambda}} (X_{\leq \lambda})
		\IsCanonicallyIsomorphicTo
		\Exit A (X)
	\]
	exhibits \( \Exit A (X) \) as a colimit of the diagram
	\( \{ \Exit {A_{\leq \lambda}}
	(X_{\leq \lambda})\}_{\lambda \in \Lambda} \)
	in the ∞-category of ∞-categories.
\end{proposition}

\begin{proof}
	We notice that
	\( \Exit {A_{\leq \lambda}} (X_{\leq \lambda})
	\to \Exit A (X) \)
	is a monomorphism for every \( \lambda \in \Lambda \),
	thus the
	canonical map \( \varinjlim_{\lambda \in \Lambda}
	\Exit {A_{\leq \lambda}} (X_{\leq \lambda}) \to \Exit A (X) \)
	is also a monomorphism.
	It is an isomorphism because, as \( \Lambda \) is filtered,
	any map of posets 
	\( \{0 < \dots < n\} \to A \) must factor
	trough \( A_{\leq \lambda} \) for some \( \lambda \in \Lambda \)
	and by definition,
	an \( n \)\=/simplex of \( \Exit A (X) \) is
	a stratified map
	\( \GeometricRealisation {\Simplex n} \to X \), which must
	then factor through \( X_{\leq \lambda} \) for some
	\( \lambda \in \Lambda \).

	To show that \( \Exit A (X) \) is an ∞-category, we need
	to check that it has the right lifting property
	\[
		\begin{tikzcd}[ampersand replacement=\&,
			column sep = huge]
			\Lambdaup^n_i
			\arrow[r, ""]
			\arrow[d, hook]
			\& \Exit A (X) \IsCanonicallyIsomorphicTo
			\varinjlim_{\lambda \in \Lambda}
			\Exit {A_{\leq \lambda}} (X_{\leq \lambda})
			\arrow[d, ""] \\
			\Simplex n
			\ar[ur, dashed]
			\arrow[r, "", swap]
			\& \ast
		\end{tikzcd}
	\]
	with respect to inner horns inclusions.
	Because inner horns \( \Lambdaup^n_i \)
	are finite simplicial set
	for every \( n < \OrdinalOmega \) and every \( 0 < i < n \),
	and \( \Lambda \) is filtered,
	any map \( \Lambdaup^n_i \to \Exit A (X) \) must factor
	through \( \Exit {A_{\leq \lambda}} (X_{\leq \lambda}) \)
	for some \( \lambda \in \Lambda \).
	By assumption \( X_{\leq \lambda} \)
	is conically \( A_{\leq \lambda} \)-stratified and a lift to
	a map
	\( \Simplex n \to \Exit {A_{\leq \lambda}} (X_{\leq \lambda})
	\subset \Exit A (X) \) exists by the previous theorem.

	Since marked simplicial sets
	form a simplicial model category
	\cite[3.1.4.4]{doi:10.1515/9781400830558},
	in order to show that \( \Exit A (X) \) is a colimit of
	the diagram 
	\( \{ \Exit {A_{\leq n}}
	(X_{\leq \lambda})\}_{\lambda \in \Lambda} \)
	in the ∞-category of ∞-categories, it will be enough to
	show that \( \Exit A (X)^\natural \) is a homotopy colimit
	of the diagram
	\( \{ \Exit {A_{\leq \lambda}}
	(X_{\leq \lambda})^\natural\}_{\lambda \in \Lambda} \)
	\cite[3.1.4.1 \& 4.2.4.1]{doi:10.1515/9781400830558}.
	This is the case since
	the model category of marked simplicial sets
	admits cofibrant generators with \( \OrdinalOmega \)-small
	domains and codomains (the marked simplicial maps, with
	underlying simplicial map
	\( \partial \Simplex n \subset \Simplex n \)),
	and filtered colimits in such
	cases coincide with homotopy colimits
	\cite[7.3]{doi:10.1006/aima.2001.2015}.
\end{proof}

\begin{corollary}
	\label{Corollary: Exit paths = Exit paths of the colimit}
	For any \( \DDO \)-space \( X \), one has an equality
	\[
		 \Exit A(X) \IsCanonicallyIsomorphicTo
		 \Exit A(X_{< \OrdinalOmega}) 
	\]
	between the exit paths ∞-categories of \( X \) and
	\( X_{< \OrdinalOmega} \).
\end{corollary}

In other words,
the exit paths ∞-category does not depend on the global
topology of the space.

\subsection{Representation theorem}

Let us recall how the ∞-category
\( \Fun(\Exit A (X), \Spaces) \) represents
\( A \)-constructible sheaves in the case where \( X \) is a
\( \DD \)-space \cite[A.10]{arXiv:0911.0018}.
First the ∞-category of functors
\( \Fun(\Exit A (X), \Spaces) \) can be replaced
by \( \Nerve (\SliceExit_X^\circ) \), the ∞-category associated
to the simplicially enriched category of fibrant-cofibrant objects
of the category
\( \Slice {\SSets} {\Exit A (X)} \) endowed with the covariant
model structure, via
\[
	\begin{tikzcd}
		\Fun(\Exit A (X), \Spaces) 	
		&
		\lar
		\Nerve\left(
		\left(
			\SSets^{\CC[\Exit A (X)]}
		\right)^\circ\right)
		\rar
		&
		\Nerve\left(\SliceExit_X^\circ\right)
	\end{tikzcd}
\]
a chain of equivalences
\cite[2.2.1.2 \& 4.2.4.4]{doi:10.1515/9781400830558},
where the left functor is a forgetful functor and the right
functor is the unstraightening functor.
One can then define a functor
\( \Opens X \Op \times \SliceExit_X^\circ
\to \SSets^\circ \)
\[
	(U,Y) \longmapsto
	\Fun_{\Exit A (X)}(\Exit A (U), Y)
\]
which will induce a functor
\[
	\begin{tikzcd}
		\Nerve\left(\SliceExit_X^\circ\right)
		\rar["\ConstructibleRealisation_X"]
			& \Presheaves(X)
	\end{tikzcd}
\]
with values in the ∞-category of presheaves on \( X \).

\begin{theorem}[{\cite[A.10.5, A.10.10 \& A.10.3]{arXiv:0911.0018}}]
	\label{Theorem: representation for D-spaces}
	Let \( X \to A \) be a \( \DD \)-space.
	The functor
	\( \ConstructibleRealisation_X
	\From \Nerve(\SliceExit_X^\circ) \to \Presheaves(X) \)
	is fully faithful and its image is
	equivalent to the subcategory of
	\( A \)-constructible sheaves on \( X \).
\end{theorem}

We shall now extend this theorem to the case where \( X \) is a
\( \DDO \)-space.

\begin{lemma}
	\label{Lemma: Psi is right adjoint}
	Let \( X \) be an \( A \)-stratified space, then
	\( \Nerve(\SliceExit_X^\circ) \) is a presentable
	∞\=/category and the functor
	\( \ConstructibleRealisation_X \) is a right adjoint.
\end{lemma}

\begin{proof}
	The simplicial model category \( \SliceExit_X \) is
	combinatorial \cite[2.1.4.6]{doi:10.1515/9781400830558},
	so its associated
	∞-category is presentable
	\cite[A.3.7.6]{doi:10.1515/9781400830558}.

	In view of the adjoint functor theorem
	\cite[5.5.2.9]{doi:10.1515/9781400830558}, it shall
	be enough to show that \( \ConstructibleRealisation_X \)
	is both continuous and accessible.

	For every open subset \( U \subset X \),
	the functor
	\( Y \mapsto \Fun_{\Exit A (X)}(\Exit A (U), Y) \)
	from \( \SliceExit_X^\circ \) to \( \SSets^\circ \)
	preserves homotopy limits since \( \SliceExit_X \) is
	a simplicial model category
	and \( \Exit A (U) \) is cofibrant (as all objects are).
	As a consequence the functor
	\( Y \mapsto \ConstructibleRealisation_X(Y)(U) \) preserves
	all small limits
	\cite[4.2.4.1 \& 4.2.3.14]{doi:10.1515/9781400830558}.
	It follows that \( \ConstructibleRealisation_X \) preserves
	all small limits because limits in presheaves ∞-categories
	are computed pointwise
	\cite[5.1.2.3]{doi:10.1515/9781400830558}.

	Using the same arguments, to prove that
	\( \ConstructibleRealisation_X \) is accessible, it
	shall be enough to find a cardinal \( \kappa \)
	such that the functors
	\( Y \mapsto \ConstructibleRealisation_X(Y)(U) \)
	commute with \( \kappa \)\=/filtered homotopy colimits
	for every \( U \subset X \).
	Since \( \SliceExit_X \) is combinatorial,
	there exists a cardinal \( \kappa \) such that
	all \( \kappa \)-filtered colimits be
	homotopy colimits
	\cite[7.3]{doi:10.1006/aima.2001.2015}.
	Enlarging \( \kappa \) if necessary, we can
	also demand that
	\( \Exit A (U) \) be \( \kappa \)-small
	for every open subset \( U \subset X \).
	As a consequence,
	\( Y \mapsto \ConstructibleRealisation_X(Y)(U) \)
	commutes with all (homotopy)
	\( \kappa \)-filtered colimits
	for every open \( U \subset X \).
\end{proof}

\begin{lemma}
	\label{Lemma: j(n)^* j(n)_* adjunction}
	Let \( X \) be a \( \DDO \)-space.
	For each \( n < \OrdinalOmega \), denote by
	\( j(n) \) the inclusion of \( X_{\leq n} \)
	into \( X \).
	One has a bireflexive simplicial model localisation
	\[
		\begin{tikzcd}[ampersand replacement=\&, column sep = huge]
			\SliceExit_X
			\arrow[r,"j(n)^\ast" description]
			\&
			\SliceExit_{X_{\leq n}}
			\arrow[l, shift right=3, "j(n)_!" swap, hook]
			\arrow[l, shift left=3, "j(n)_\ast", hook']
		\end{tikzcd}
	\]
	where \( j(n)_! \) is the forgetful functor, where
	\[
		j(n)^\ast(Y) \coloneqq
		\Exit {A_{\leq n}} (X_{\leq n})
		\times_{\Exit A (X)} Y
	\]
	and where both categories are endowed with the covariant model
	structure.
\end{lemma}

\begin{proof}
	The existence of the right adjoint to \( j(n)_\ast \) can
	be obtained by the adjoint functor theorem,
	using the presentability
	of \( \SliceExit_X \) and the universality of colimits
	in \( \SliceExit_{X_{\leq n}} \).
	The obvious fully faithfulness of the left adjoint \( j(n)_! \)
	implies the fully faithfulness of the right adjoint
	\( j(n)_\ast \).

	The forgetful functor \( j(n)_! \) has an obvious
	simplicial enrichment.
	Moreover, one has
	evident isomorphisms
	\( j(n)_!(\Simplex p \times Y) \IsCanonicallyIsomorphicTo
	\Simplex p \times j(n)_!(Y) \) natural in \( Y \in
	\SliceExit_{X_{\leq n}} \) for every \( p < \OrdinalOmega \).
	This is enough to endow \( j(n)^\ast \)
	and the adjunction
	\( j(n)_! \IsAdjointTo j(n)^\ast \)
	with a simplicial enrichment
	\cite[3.7.10]{doi:10.1017/cbo9781107261457}.
	Likewise
	one has isomorphisms \( j(n)^\ast (\Simplex p \times Y)
	\IsCanonicallyIsomorphicTo \Simplex p \times j(n)^\ast (Y) \)
	natural in \( Y \in \SliceExit_X \)
	for every \( p < \OrdinalOmega \), giving \( j(n)_\ast \)
	and the adjunction \( j(n)^\ast \IsAdjointTo j(n)_\ast \)
	a simplicial enrichment.

	In this kind of setup, the pair
	\( j(n)_! \IsAdjointTo j(n)^\ast \) is always a model
	adjunction
	\cite[2.1.4.10]{doi:10.1515/9781400830558}.
	The embedding \( \Exit {A_{\leq n}} (X_{\leq n}) \subset
	\Exit A (X) \) is a right fibration.
	This follows from the fact that if a stratified path
	ends in \( X_{\leq n} \), then all the path must be in
	\( X_{\leq n} \).
	As a consequence, the adjunction
	\( j(n)^\ast \IsAdjointTo j(n)_\ast \) is a model
	adjunction \cite[11.2]{Joyal}.
\end{proof}

As a consequence, both functors
\( j(n)^\ast \) and \( j(n)_\ast \) preserve fibrant objects
by the previous lemma and induce a reflexive localisation
\[
	\begin{tikzcd}[ampersand replacement=\&, column sep = huge]
		\Nerve\left(\SliceExit_X^\circ\right)
		\arrow[r, shift left=2,"j(n)^\ast"]
		\&
		\arrow[l, shift left=2, "j(n)_\ast", hook']
		\Nerve\left(\SliceExit_{X_{\leq n}}^\circ\right)
	\end{tikzcd}
\]
between the associated ∞-categories
\cite[5.2.4.5]{doi:10.1515/9781400830558}.
These adjunctions induce an adjunction
\[
	\begin{tikzcd}[ampersand replacement=\&, column sep = huge]
		\Nerve\left(\SliceExit_X^\circ\right)
		\arrow[r, shift left=2,"j^\ast"]
		\&
		\arrow[l, shift left=2, "j_\ast"]
		\varprojlim_{n < \OrdinalOmega}
		\Nerve\left(\SliceExit_{X_{\leq n}}^\circ\right)
	\end{tikzcd}
\]
where \( j^\ast \) is given by the family of the functors
\( j(n)^\ast \) and where
\[
	\textstyle
	j_\ast(\{Y_{\leq n}\})
	\coloneqq \varprojlim_{n < \OrdinalOmega} j(n)_\ast Y_{\leq n}
\]
for any \( \{Y_{\leq n}\}
\in \varprojlim_{n < \OrdinalOmega}
\Nerve(\SliceExit_{X_{\leq n}}^\circ) \).

\begin{lemma}
	\label{Lemma: comparison}
	Let \( \Category C \) be an ∞-category and consider two
	towers of monomorphisms
	\( X_0 \hookrightarrow X_1 \hookrightarrow \dots
	\hookrightarrow X_n \hookrightarrow \dots \hookrightarrow X \)
	and 
	\( Y_0 \hookrightarrow Y_1 \hookrightarrow \dots
	\hookrightarrow Y_n \hookrightarrow \dots \hookrightarrow Y \)
	in \( \Category C \).
	Assume that we have also given
	a morphism \( X \to Y \)
	and equivalences \( X_n \IsEquivalentTo Y_n \) together with
	natural transformations making the square
	\[
		\begin{tikzcd}[ampersand replacement=\&]
			X_n
			\arrow[r, "\IsEquivalentTo"]
			\arrow[d, hook]
			\& Y_n
			\arrow[d, hook] \\
			X
			\arrow[r]
			\& Y
		\end{tikzcd}
	\]
	commute, for every \( n < \OrdinalOmega \).
	This data is enough to induce a commutative square
	\[
		\begin{tikzcd}[column sep = large]
			\varinjlim_{n < \OrdinalOmega} X_n
			\arrow[r, "\IsEquivalentTo"]
			\arrow[d]
			& \varinjlim_{n < \OrdinalOmega} Y_n
			\arrow[d] \\
			X
			\arrow[r]
			& Y
		\end{tikzcd}
	\]
	whose top map is an equivalence and whose
	vertical maps are the canonical projections.
\end{lemma}

\begin{proof}
	By composition with \( X \to Y \), we get a commutative
	triangle
	\[
		\begin{tikzcd}[column sep = large]
			\varinjlim_{n < \OrdinalOmega} X_n
			\arrow[d]
			\ar[dr]
			& \\
			X
			\arrow[r]
			& Y
		\end{tikzcd}
	\]
	letting us reduce to the case where \( X = Y \).
	
	Up to an equivalence, we can assume that the ∞-category
	of subobjects of \( Y \) is a
	\( 0 \)-category
	\cite[2.3.4.18]{doi:10.1515/9781400830558}.
	We thus reduce to the case where
	the maps \( X_n \hookrightarrow Y \) and
	\( Y_n \hookrightarrow Y \) are equal for
	every \( n < \OrdinalOmega \) and the result is trivial.
\end{proof}

\begin{remark}
	\label{Remark: monomorphisms}
	We shall use this lemma in the special case where
	\( \Category C \) is the ∞\=/category of presentable
	∞-categories and right adjoints.
	In this case fully faithful right adjoint functors are
	monomorphisms
	\cite[5.7]{doi:10.1016/j.aim.2018.05.031}.
	As a consequence, reflexive localisation functors
	are epimorphisms in the ∞-category of presentable
	∞-categories and left adjoint functors
	\cite[5.5.3.4]{doi:10.1515/9781400830558}.
\end{remark}

\begin{proposition}[(Dévissage)]
	\label{Proposition: dévissage of A}
	Let \( X \) be a \( \DDO \)-space.
	The adjunction
	\[
		\begin{tikzcd}[ampersand replacement=\&, column sep = huge]
			\Nerve\left(\SliceExit_X^\circ\right)
			\arrow[r, shift left=2,"j^\ast"]
			\&
			\arrow[l, shift left=2, "j_\ast"]
			\varprojlim_{n < \OrdinalOmega}
			\Nerve\left(\SliceExit_{X_{\leq n}}^\circ\right)
		\end{tikzcd}
	\]
	induced by the pairs \( j(n)^\ast \IsAdjointTo j(n)_\ast \)
	is an equivalence of ∞-categories.
\end{proposition}

\begin{proof}
	Let \( \SSets^{\CC[\Exit A (X)]} \) denote the category
	of simplicially enriched functors from
	\( \CC[\Exit A (X)] \) to \( \SSets \) endowed with the
	projective model structure.
	The diagram of right model adjoints
	\[
		\begin{tikzcd}[ampersand replacement=\&,
			column sep = 90]
			\SSets^{\CC[\Exit A (X)]} 
			\arrow[r, "\text{unstraightening}"]
			\arrow[d, "j(n)^\ast", swap]
			\& 
			\SliceExit_X
			\arrow[d, "j(n)^\ast"] \\
			\SSets^{\CC[\Exit A (X_{\leq n})]} 
			\arrow[r, "\text{unstraightening}"]
			\& 
			\SliceExit_{X_{\leq n}}
		\end{tikzcd}
	\]
	is commutative
	\cite[2.2.1.1]{doi:10.1515/9781400830558}.

	Moreover, all functors are simplicially enriched
	\cite[2.2.2.12]{doi:10.1515/9781400830558}
	and the commutation holds at the enriched level.
	Let us see why.
	The enrichment for the unstraightening functor
	stems from the following sequence:
	for any two simplicial functors
	\( F \) and \( G \) and any simplicial set \( K \),
	any natural transformation
	\( F \boxtimes K \to G \)
	induces a natural transformation
	\( \alpha(F, G, K) \)
	\[
		\Un(F) \times K \longrightarrow
		\Un(F) \times \Sing_{\mathrm{Q}^\bullet}(K)
		\IsCanonicallyIsomorphicTo
		\Un(F \boxtimes K) \longrightarrow \Un(G)
	\]
	while the enrichment for both pullback functors comes
	straight from commutation with tensoring with simplicial
	sets.
	Using this fact, one gets that
	\( j(n)^\ast \alpha(F, G, K) \) is
	isomorphic
	to \( \alpha(j(n)^\ast F, j(n)^\ast G, K) \).

	Because
	\( \Exit {A_{\leq n}} (X_{\leq n}) \subset \Exit A (X) \)
	is a monomorphism, the associated simplicial functor
	\( \CC[\Exit {A_{\leq n}} (X_{\leq n})]
	\to \CC[\Exit A (X)] \) is a cofibration
	\cite[2.2.5.1]{doi:10.1515/9781400830558},
	as a consequence the
	associated pullback functor \( j(n)^\ast \) preserves
	cofibrant objects \cite[A.3.3.9]{doi:10.1515/9781400830558}.
	Then, all functors in this square preserve fibrant-cofibrant
	objects and one gets
	\[
		\begin{tikzcd}
			\Nerve\left(\left(
			\SSets^{\CC[\Exit A (X)]}\right)^\circ\right)
			\arrow[r, equal]
			\arrow[d, "j(n)^\ast", swap]
			& 
			\Nerve\left(\SliceExit_X^\circ\right)
			\arrow[d, "j(n)^\ast"] \\
			\Nerve\left(\left(\SSets^{
			\CC[\Exit A (X_{\leq n})]}\right)^\circ\right)
			\arrow[r, equal]
			& 
			\Nerve\left(\SliceExit_{X_{\leq n}}^\circ\right)
		\end{tikzcd}
	\]
	an equivalence between the two induced pullback functors
	at the ∞-category level
	\cite[2.2.3.11]{doi:10.1515/9781400830558}.
	One also has a commutative diagram
	\[
		\begin{tikzcd}
			\Nerve\left(\left(
			\SSets^{\CC[\Exit A (X)]}\right)^\circ\right)
			\arrow[r, equal]
			\arrow[d, "j(n)^\ast", swap]
			& 
			\Fun(\Exit A (X), \Spaces)
			\arrow[d, "j(n)^\ast"] \\
			\Nerve\left(\left(\SSets^{
			\CC[\Exit A (X_{\leq n})]}\right)^\circ\right)
			\arrow[r, equal]
			& 
			\Fun\left(
			\Exit {A_{\leq n}} (X_{\leq n}), \Spaces\right)
		\end{tikzcd}
	\]
	letting us identify \( j(n)^\ast \) with the pullback
	at the level of ∞-categories of functors
	\cite[4.2.4.4]{doi:10.1515/9781400830558}.

	We then reduce to show that the canonical map
	\[
		\begin{tikzcd}
			\Fun(\Exit A (X), \Spaces)
			\rar["j^\ast"]
			&
			\varprojlim_{n < \OrdinalOmega}
			\Fun\left(
			\Exit {A_{\leq n}} (X_{\leq n}), \Spaces\right)
		\end{tikzcd}
	\]
	is an equivalence
	\DoubleUnskipRef{Lemma: comparison}{Remark: monomorphisms},
	which follows from the fact that
	\( \Exit A (X) \) is a colimit
	of the diagram
	\(\{\Exit {A_{\leq n}} (X_{\leq n})\}_{n < \OrdinalOmega} \)
	in the ∞-category of ∞-categories
	\UnskipRef{Proposition: Exit of DDO-space}.
\end{proof}

\begin{theorem}
	Let \( X \to A \to \OrdinalOmega \) be a
	\( \DDO \)-space and
	assume that \( X \) be either a conical or a colimit
	\( \DDO \)-space.
	
	The functor \( \ConstructibleRealisation_X \) is fully
	faithful and induces
	an equivalence of ∞-categories
	\[
		\Nerve\left(\SliceExit_X^\circ\right) 
		\IsCanonicallyIsomorphicTo
		\ConstructibleHypersheaves A (X)
	\]
	between \( \Nerve(\SliceExit_X^\circ) \) and the
	∞-category of \( A \)-constructible hypersheaves on \( X \).
\end{theorem}

\begin{proof}
	For each \( n < \OrdinalOmega \) and each open subset
	\( U \subset X \)
	one has \( j(n)^\ast \Exit A (U)
	= \Exit {A_{\leq n}} (U \cap X_{\leq n}) \), implying that
	for each left fibration
	\( Y \) over \( \Exit {A_{\leq n}} (X_{\leq n}) \), 
	\( \thetaup(U, j(n)_\ast Y)
	\IsCanonicallyIsomorphicTo
	\thetaup(U \cap X_{\leq n}, Y) \) since
	\( j(n)^\ast \IsAdjointTo j(n)_\ast \) is a simplicially
	enriched adjunction
	\UnskipRef{Lemma: j(n)^* j(n)_* adjunction}.
	As a consequence, the square
	\[
		\begin{tikzcd}[column sep = huge]
			\Nerve\left(\SliceExit_X\right)
			\rar["\ConstructibleRealisation_X"]
			&
			\Presheaves(X)
			\\
			\Nerve\left(\SliceExit_{X_{\leq n}}\right)
			\ar[u, "j(n)_\ast", hook']
			\rar["\ConstructibleRealisation_{X_{\leq n}}"]
			&
			\ar[u, "j(n)_\ast"', hook]
			\Presheaves(X_{\leq n})
		\end{tikzcd}
	\]
	commutes up to a natural equivalence.
	In fact, because \( X_{\leq n} \) is a \( \DD \)-space
	and \( j(n)_\ast \) preserves sheaves among presheaves and
	constructible sheaves among sheaves
	\UnskipRef{Lemma: push of constructible hypersheaf},
	we have an induced
	\[
		\begin{tikzcd}[column sep = huge]
			\Nerve\left(\SliceExit_X\right)
			\rar["\ConstructibleRealisation_X"]
			&
			\Presheaves(X)
			\\
			\Nerve\left(\SliceExit_{X_{\leq n}}\right)
			\ar[u, "j(n)_\ast", hook']
			\rar["\ConstructibleRealisation_{X_{\leq n}}"]
			&
			\ar[u, "j(n)_\ast"', hook]
			\ConstructibleSheaves {A_{\leq n}} (X_{\leq n})
		\end{tikzcd}
	\]
	commutative square of right adjoints
	\UnskipRef{Lemma: Psi is right adjoint}
	by the representation theorem for \( \DD \)-spaces
	\UnskipRef{Theorem: representation for D-spaces}.

	One thus gets a commutative diagram
	\[
		\begin{tikzcd}[column sep = 90]
			\Nerve\left(\SliceExit_X\right)
			\rar["\ConstructibleRealisation_X"]
			&
			\Presheaves(X)
			\\
			\varinjlim_{n < \OrdinalOmega}
			\Nerve\left(\SliceExit_{X_{\leq n}}\right)
			\ar[u, "j_\ast"]
			\rar["\varinjlim_{n < \OrdinalOmega}
			\ConstructibleRealisation_{X_{\leq n}}"]
			&
			\ar[u, "j_\ast"']
			\varinjlim_{n < \OrdinalOmega}
			\ConstructibleSheaves {A_{\leq n}} (X_{\leq n})
		\end{tikzcd}
	\]
	where the colimits are taken in the ∞-category of
	presentable ∞-categories and right adjoints
	\DoubleUnskipRef{Lemma: comparison}{Remark: monomorphisms}.
	By the representation theorem for \( \DD \)-spaces
	\UnskipRef{Theorem: representation for D-spaces},
	the bottom map is an equivalence of ∞-categories.
	By dévissage
	\UnskipRef{Proposition: dévissage of A},
	the left arrow is an equivalence of ∞\=/categories.
	Indeed, colimits in the ∞-category of presentable ∞-categories
	and right adjoints are canonically equivalent to
	the associated limit
	in the ∞-category of presentable ∞-categories
	and left adjoints
	\cite[5.5.3.4]{doi:10.1515/9781400830558}.
	Using the same argument, one deduces that
	the right arrow is fully faithful and its
	image is the subcategory of
	\( A \)-constructible hypersheaves again by dévissage
	\DoubleUnskipRef{Theorem: dévissage colimit}{Theorem:
	dévissage conical}.
\end{proof}

\begin{corollary}[(Representation theorem, conical case)]
	\label{Representation, conical case}
	Let \( X \to A \to \OrdinalOmega \) be a conical
	\( \DDO \)-space.
	Then the ∞-category of exit paths
	\( \Exit A (X) \)
	represents
	\[
		\Fun(\Exit A (X), \Spaces)
		\IsCanonicallyIsomorphicTo
		\ConstructibleHypersheaves A (X)
	\]
	the
	∞-category of \( A \)\=/constructible hypersheaves on \( X \).
\end{corollary}

\begin{corollary}[(Representation theorem, colimit case)]
	\label{Representation, colimit case}
	Let \( X \to A \to \OrdinalOmega \)
	be a colimit \( \DDO \)-space.
	The ∞-category of exit paths
	\( \Exit A (X) \)
	\[
		\Fun(\Exit A (X), \Spaces)
		\IsCanonicallyIsomorphicTo
		\ConstructibleHypersheaves A (X)
		\IsCanonicallyIsomorphicTo
		\ConstructibleSheaves A (X)
	\]
	represents the ∞-category of \( A \)-contructible sheaves.
\end{corollary}

\begin{remark}
	\label{Remark: functoriality of representation theorem}
	It is natural to conjecture that the equivalences
	\( \ConstructibleRealisation_X \) are part of a natural
	equivalence of functors.
	In particular, we should have commutative squares
	\[
		\begin{tikzcd}[column sep = huge]
			\Fun(\Exit A (Y), \Spaces)
			\arrow[d, "\rotatebox{270}{\( \IsEquivalentTo \)}"]
			\arrow[r, "\Exit A (f)^\ast"]
			& 
			\Fun(\Exit A (X), \Spaces)
			\arrow[d, "\rotatebox{270}{\( \IsEquivalentTo \)}"] \\
			\ConstructibleHypersheaves A (Y)
			\arrow[r, "\Hyperpullback f"]
			& 
			\ConstructibleHypersheaves A (X)
		\end{tikzcd}
	\]
	whenever given a conical or colimit
	\( \DDO \)-space \( Y \to A \to \OrdinalOmega \)
	and a continuous map \( f \From X \to Y \) letting
	\( X \to Y \to A \to \OrdinalOmega \) be a
	conical or colimit \( \DDO \)-space.

	The existence of these commutative squares has been
	shown by Lurie in the case where \( X \) and \( Y \)
	are \( \DD \)-spaces
	\cite[A.10.16]{arXiv:0911.0018}.
	Using a similar proof, one can show their existence
	in the case of conical or colimit \( \DDO \)-spaces.
\end{remark}

\subsection{Homotopy invariance of exit paths}

Ayala, Francis and Rozenblyum have shown that the exit paths
∞-category functor is fully faithful between the ∞-category of
smooth stratified spaces and the ∞-category of ∞-categories
\cite{doi:10.4171/jems/856}.
We prove here a modest extension to \( \DDO \)-spaces, showing
that the exit paths functor is homotopy invariant.
This homotopy invariance result invites us to think that the
stratified homotopy hypothesis for smooth stratified spaces could be
extended to stratified spaces whose poset of strata does not
satisfy the ascending chain condition.

\begin{lemma}
	Let \( f \From \Category C \to \Category D \)
	be a functor between two small
	idempotent complete ∞-categories.
	If the pullback functor
	\[
		\begin{tikzcd}
			\Fun(\Category D, \Spaces)
			\rar["f^\ast"]
				& \Fun(\Category C, \Spaces)
		\end{tikzcd}
	\]
	is an equivalence of ∞-categories, then
	so is \( f \).
\end{lemma}

\begin{proof}
	Let \( f_! \) denote the left adjoint to \( f^\ast \).
	The following square
	\[
		\begin{tikzcd}[ampersand replacement=\&]
			\Fun(\Category C, \Spaces)
			\arrow[r, "f_!"]
			\& \Fun(\Category D, \Spaces) \\
			\Category C\Op
			\arrow[r, "f\Op"]
			\ar[u, hook, "\Yoneda_{\Category C}"]
			\& \Category D\Op
			\ar[u, hook, "\Yoneda_{\Category D}"]
		\end{tikzcd}
	\]
	is commutative
	\cite[5.2.6.3]{doi:10.1515/9781400830558}.
	By assumption,
	\( f^\ast \) is an equivalence of ∞\=/categories
	and so is its left adjoint \( f_! \).
	Then, \( f_! \) induces an equivalence of ∞\=/categories
	between the respective full subcategories of
	completely compact objects.
	By idempotent completeness, these full subcategories
	are equivalent to the full subcategories of representable
	functors
	\cite[5.1.6.8]{doi:10.1515/9781400830558}.
	As a consequence \( f\Op \) and thus \( f \) are equivalences
	of ∞-categories.
\end{proof}

\begin{lemma}
	Let \( X \) be an \( A \)-stratified space,
	then \( \Exit A (X) \) is an idempotent complete ∞-category.
\end{lemma}

\begin{proof}
	An idempotent stratified path in \( X \) must stay in the same
	stratum, in which case it is invertible and thus trivial.
\end{proof}

\begin{theorem}
	Let \( \Lambda \) be a filtered poset,
	\( A \) be a \( \Lambda \)-stratified poset,
	\( Y \) be an \( A \)\=/stratified space
	and let \( f \From X \to Y \) be a continuous
	map.

	Assume that for every \( \lambda \in \Lambda \),
	both \( X_{\leq \lambda} \to A_{\leq \lambda} \)
	and \( Y_{\leq \lambda} \to A_{\leq \lambda} \)
	be \( \DD \)-spaces.
	If \( f \) is an \( A \)-stratified homotopy equivalence,
	then the functor
	\[
		\begin{tikzcd}[column sep = huge]
			\Exit A (X)
			\rar["\Exit A (f)"]
				& \Exit A (Y)
		\end{tikzcd}
	\]
	is an equivalence of ∞-categories.
\end{theorem}

\begin{proof}
	Since \( f \) is an \( A \)-stratified homotopy equivalence,
	it induces an \( A_{\leq \lambda} \)\=/stratified
	homotopy equivalence \( f_{\leq \lambda}
	\From X_{\leq \lambda} \IsEquivalentTo Y_{\leq \lambda} \)
	for each \( \lambda \in \Lambda \).
	By homotopy invariance, the pullback functor
	\( (f_{\leq \lambda})^\ast
	\From \ConstructibleSheaves {A_{\leq \lambda}}(Y_{\leq \lambda})
	\to
	\ConstructibleSheaves {A_{\leq \lambda}}(X_{\leq \lambda}) \)
	is an equivalence of ∞\=/categories
	\TripleUnskipRef{Corollary: hyperconstructible = constructible}
	{Proposition: D-spaces}
	{theorem: homotopy invariance of hyperconstructibles}.
	As a consequence, the pullback functor
	\( \Exit {A_{\leq \lambda}}(f_{\leq \lambda})^\ast \) from
	\( \Fun(\Exit {A_{\leq \lambda}} (Y_{\leq \lambda}), \Spaces) \)
	to \( \Fun(\Exit {A_{\leq \lambda}}
	(X_{\leq \lambda}), \Spaces) \) is also
	an equivalence of ∞\=/categories
	\UnskipRef{Remark: functoriality of representation theorem}.
	Then by the previous lemmas,
	\( \Exit {A_{\leq \lambda}} (f_{\leq \lambda}) \)
	is an equivalence between the ∞-category
	\( \Exit {A_{\leq \lambda}} (X_{\leq \lambda}) \)
	and
	\( \Exit {A_{\leq \lambda}} (Y_{\leq \lambda}) \).
	Taking the colimit
	\UnskipRef{Proposition: Exit of DDO-space},
	we get that
	\( \Exit A (f) \) is an equivalence of ∞-categories.
\end{proof}

\section{Examples of application}

We shall give examples of applications of the representation theorem
for constructible hypersheaves in the case where the stratifying poset
does not satisfy the ascending chain condition.
Simplicial complexes and \( \CW \)-complexes are natural examples
of colimit \( \DDO \)-spaces.
We shall then describe a conical example of \( \DDO \)-space, the metric
exponential.

One common point of these examples is that each stratum is locally
homeomorphic to a locally convex topological vector space
and we shall need to know that such spaces are locally of singular
shape.

\begin{lemma}
	\label{Lemma: locally of singular shape}
	Let \( X \) be a topological space
	locally homeomorphic
	to a locally convex real topological vector space
	\( V \).
	Then \( X \) is locally of singular shape.
\end{lemma}

\begin{proof}
	By assumption \( X \) is covered by spaces homeomorphic
	to \( V \), so it shall be enough to show that
	\( V \) is locally of singular shape
	\cite[A.4.16]{arXiv:0911.0018}.
	Let \( U \subset V \) be an open subset of \( V \) and
	let \( A \) be the set of all convex open subsets of \( U \).
	Since
	\( V \) is locally convex \( U = \cup_{C \in A} C \).
	Moreover for any finite \( B \subset A \), the intersection
	\( \cap_{C \in B} C \) is still convex.
	Lastly, any convex subset in a topological vector space
	is contractible and thus of singular shape
	\cite[A.4.11]{arXiv:0911.0018}.
	This is enough to show that \( U \) is also of singular
	shape
	\cite[A.4.14]{arXiv:0911.0018}.
\end{proof}

\subsection{Simplicial complexes}

Let \( X \) be a simplicial complex.
Its constitutive simplices naturally give \( X \) a structure
of stratified space \( X \to S \).

Lurie has shown that when \( X \) is locally finite, then
it is conically \( S \)\=/stratified
\cite[A.7.3]{arXiv:0911.0018}.
The poset of simplices \( S \) also satisfies the ascending
chain condition.
Moreover, \( \Exit S (X) \) is always an ∞-category
\cite[A.7.4]{arXiv:0911.0018}
and the canonical map
\( \Exit S (X) \to \Nerve(S) \) is an equivalence of
∞-categories
\cite[A.7.5]{arXiv:0911.0018}.

If \( X \) is a locally finite simplicial
complex,
it is then a \( \DD \)-space and the representation theorem
applies.
We shall extend this to the class of all locally countable
simplicial complexes.

\begin{definition}
	We shall say that a simplicial complex \( X \) is locally
	countable if each vertex \( \nu \in X \) belongs
	to at most a countable number of edges of \( X \).
\end{definition}

\begin{theorem}
	Let \( X \) be a locally countable
	simplicial complex with
	poset of simplices \( S \).
	One can represent
	\( S \)-constructible sheaves on \( X \)
	\[
		\Fun(\Nerve(S), \Spaces)
		\IsCanonicallyIsomorphicTo
		\ConstructibleSheaves S (X)
	\]
	as functors from \( \Nerve(S) \) to \( \Spaces \).
\end{theorem}

\begin{proof}
	For each vertex \( \nu \) let \( \Edges^\nu \) be the set
	of edges \( \epsilon \) for which \( \nu \in \epsilon \).
	By assumption \( \Edges^\nu \) is countable
	so one can find an exhaustion by finite subsets
	\[
		\Edges_{\leq 0}^\nu
		\subset \cdots \subset \Edges_{\leq n}^\nu
		\subset \cdots \subset \Edges^\nu
	\]
	for each vertex \( \nu \).
	For each \( n < \OrdinalOmega \),
	let \( X_{\leq n} \subset X \) denote
	the biggest simplicial subcomplex
	for which each edge \( \epsilon \in X_{\leq n} \) is such that
	\( \nu \in \epsilon \implies \epsilon \in \Edges_{\leq n}^\nu \).
	By construction the poset \( S_{\leq n} \)
	of simplices of \( X_{\leq n} \) is a downward closed
	subposet of \( S \) and \( X \to S \)
	is exhausted by the subcomplexes
	\( X_{\leq n} \to S_{\leq n} \) which are all \( \DD \)-spaces
	as discussed earlier.

	Every simplicial complex is paracompact and each
	simplex is locally of singular shape
	\UnskipRef{Lemma: locally of singular shape}.
	Thus \( X \to S \) is a colimit \( \DDO \)-space and
	the representation theorem applies
	\UnskipRef{Representation, colimit case}.
\end{proof}

\subsection{CW-complexes}

The case of \( \CW \)-complexes resembles the one of simplicial
complexes: the set of cells \( C \) of a \( \CW \)-complex
\( X \) can be made into a poset by letting
\( d \leq c \) whenever \( d \subset \overline c \).
However, this poset is generally odd with cells attached to other
cells but isolated in the poset.
It is then natural to restrict one's attention to normal
\( \CW \)-complexes.

\begin{definition}
	A \( \CW \)-complex \( X \) with set of cells
	\( C \) is said to be normal if for every \( c \in C \),
	its boundary is a again a union of cells of \( X \).

	We shall say that a normal \( \CW \)-complex
	\( X \) is locally finite whenever \( C_{\nu <} \) is
	finite for every \( 0 \)-cell \( \nu \), and
	we shall say that \( X \) is locally countable if instead
	\( C_{\nu <} \) is countable.
\end{definition}

The last remaining obstacle is the conicality of the stratification as
locally finite normal \( \CW \)-complexes
might not be conically stratified.
Tamaki has introduced the notion of cylindrically normal
\( \CW \)-complex
\cite[4.1]{doi:10.1142/9789813226579_0006}
and with Tanaka, they have shown that for finite cylindrically
normal \( \CW \)-complexes, the cell stratification is
conical
\cite[1.7]{doi:10.1007/978-981-13-5742-8_15}.
Since conicality is a local condition, this immediately extends
to locally finite normal \( \CW \)-complexes.
Regular \( \CW \)\=/complexes and, real or complex  projective spaces are
examples of cylindrically normal \( \CW \)-complexes
\cite[4.3]{doi:10.1142/9789813226579_0006}.

\begin{theorem}
	Let \( X \) be a locally countable cylindrically normal
	\( \CW \)-complex with poset of cells \( C \).
	Then \( \Exit C (X) \) is an ∞-category
	\[
		\Fun(\Exit C (X), \Spaces)
		\IsCanonicallyIsomorphicTo
		\ConstructibleSheaves C (X)
	\]
	representing \( C \)-constructible sheaves on \( X \).
\end{theorem}

\begin{proof}
	Every \( \CW \)-complex is paracompact and each cell
	being homeomorphic to a locally convex topological vector
	space, it is locally of singular shape
	\UnskipRef{Lemma: locally of singular shape}.
	By assumption if \( X \) is locally countable, for every
	\( 0 \)-cell \( \nu \), one can find an exhaustion of
	\( C_{\nu <} \)
	\[
		C_{\nu <}^0 \subset
		\cdots
		C_{\nu <}^n \subset
		\cdots C_{\nu <}
	\]
	by finite subsets.
	Let \( X_{\leq n} \subset X \) denote the biggest subcomplex
	such that for each cell \( e \in X_{\leq n} \), if
	\( \nu \leq e \), then \( e \in C^n_{\nu <} \).
	By construction \( X_{\leq n} \) is locally finite
	and cylindrically normal, so it is conically stratified
	by the discussion above.
	In addition, because the border of every cell \( e \in X \)
	is a finite union of cells of lower dimension, a direct
	induction shows that each cell
	\( e \) belongs to one of the \( X_{\leq n} \).
	It follows that \( X \) is a colimit \( \DDO \)-space and
	the representation theorem applies
	\UnskipRef{Representation, colimit case}.
\end{proof}

\subsection{The exponentials}
\label{Section: exponentials}

Given a set \( X \), the exponential on \( X \) is the set
\( \Exp(X) \) of all finite subset \( S \subset X \).
If \( (X, d) \) is a metric space, the formula
\[
	D(S, T) \coloneqq
	\max
		\begin{cases}
			\max_{s \in S} \min_{t \in T} d(s,t)
			\\
			\max_{t \in T} \min_{s \in S} d(s, t)
		\end{cases}
\]
canonically endows the exponential with a generalised metric.
That is, a metric space in which two points
can be at infinite distance from one another. Here
this is the case of the empty configuration:
\( D(\emptyset, S) = + ∞ \) if \( S \) is not empty.
The set \( \Exp(X) \) when endowed with
its metric \( D \) is the metric exponential \( \MetricExp(X) \).
The metric exponential is canonically stratified over the poset
\( \PointedOmega = \{0\} \amalg \{1 < 2 < \cdots \} \),
with a finite subset \( S \subset X \) being sent to
its cardinality.
Replacing the metric topology of \( \MetricExp(X) \) with the colimit
one arising from this stratification, one obtains
\[
	\textstyle
	\TopExp(X) \coloneqq \varinjlim_{n \in \PointedOmega}
	\MetricExp^{\leq n}(X)
\]
the topological exponential.
When \( M \) is a Fréchet manifold, one can show that
the metric exponential is conically stratified
\cite{Anna}.
It follows that \( \MetricExp(M) \) is a conical \( \DDO \)-space
and \( \TopExp(M) \) is its associated colimit \( \DDO \)-space.
In particular
\[
	\ConstructibleHypersheaves {\PointedOmega}
	\left(\MetricExp(M)\right)
	\IsCanonicallyIsomorphicTo
	\ConstructibleSheaves {\PointedOmega} \left(\TopExp(M)\right)
	\IsCanonicallyIsomorphicTo \Fun(\Exit {\PointedOmega}
	(\MetricExp(X)), \Spaces)
\]
both have the same ∞-category of constructible hypersheaves
and these are representable using exit paths.

Constructible cosheaves on the exponential form a key ingredient in
the theory of locally constant factorisation algebras.
Constructible hypersheaves on the metric exponential are
for example
used by Lurie, to bridge locally constant factorisation algebras on
a finite dimensional manifold \( M \) with the theory of
\( \mathsf{E}_M \)-algebras
\cite[3.6.10]{arXiv:0911.0018}.

\subsection*{Acknowledgements}

The author would like to thank Anna Cepek
for introducing him to the exit path ∞-category and Lurie's
representation theorem of constructible sheaves.

\bibliography{ms.bbl}

\end{document}